\documentclass[10pt,reqno]{amsart}
\usepackage{amsmath} 
\usepackage{amssymb}
\usepackage{dsfont}
\usepackage[dvips,draft,final]{graphics}
\usepackage{amssymb,amsmath}
\usepackage[T1]{fontenc}
\usepackage{fancyhdr}
\usepackage{url}

\usepackage{color}
\usepackage{graphicx}

\def\squarebox#1{\hbox to #1{\hfill\vbox to #1{\vfill}}}

\theoremstyle{plain}

\newtheorem{Thm}{Theorem}
   
\newtheorem{cor}{Corollary}

\newtheorem{lem}{Lemma}

\newcommand{\re}{\mathfrak R}
\newcommand{\im}{\mathfrak I}

\newcommand{\R}{\mathbb{R}}
 
\newcommand{\Bb}{\mathcal{H}_1({\R}^n)}
\newcommand{\B}{\dot{\mathcal{H}}_1({\R}^n)}

\newcommand{\CI}{{\mathcal C}^{\infty}_{0}({\R}^{n}) }

\newcommand{\C}{{\mathbb C}}

\def\epsilon{\varepsilon}
\def\phi {\varphi}

\newtheorem{rem}{Remark}
\newtheorem{prop}{Proposition}
\newtheorem{defi}{Definition}

\providecommand{\abs}[1]{\left\lvert#1\right\rvert}
\providecommand{\norm}[1]{\left\lVert#1\right\rVert}
\numberwithin{equation}{section}
\renewcommand{\d}{\textrm{d}}
\renewcommand{\leq}{\leqslant}
\renewcommand{\geq}{\geqslant}
\providecommand{\abs}[1]{\left\lvert#1\right\rvert}
\providecommand{\norm}[1]{\left\lVert#1\right\rVert}

\begin{document}

\title[Meromorphic continuation of the resolvent]
{On the meromorphic continuation of the resolvent for the wave equation with time-periodic perturbation and applications}

\author[Y. Kian]{Yavar Kian}

\address {Universit\'e Bordeaux I, Institut de Math\'ematiques de Bordeaux,  351, Cours de la Lib\'eration, 33405  Talence, France}

\email{Yavar.Kian@math.u-bordeaux1.fr}

\maketitle
\begin{abstract}
Consider the wave equation $\partial_t^2u-\Delta_xu+V(t,x)u=0$, where $x\in\R^n$ with $n\geq3$ and $V(t,x)$ is $T$-periodic in time and decays exponentially in space.  Let $ U(t,0)$ be the associated propagator and let  $R(\theta)=e^{-D\left\langle x\right\rangle}(U(T,0)-e^{-i\theta})^{-1}e^{-D\left\langle x\right\rangle}$  be the resolvent  of the Floquet operator $U(T,0)$  defined for $\im(\theta)>BT $ with $B>0$ sufficiently large. We establish a meromorphic continuation of $R(\theta)$   from which we deduce the asymptotic expansion of  $e^{-(D+\epsilon)\left\langle x\right\rangle}U(t,0)e^{-D\left\langle x\right\rangle}f$, where $f\in \dot{H}^1(\R^n)\times L^2(\R^n)$, as $t\to+\infty$  with a remainder term whose energy decays exponentially  when $n$ is odd and  a remainder term whose energy is bounded with respect to $t^l\log(t)^m$, with $l,m\in\mathbb Z$, when $n$ is even.  Then, assuming that  $R(\theta)$ has no poles lying in $\{\theta\in\C\  :\ \im(\theta)\geq0\}$ and is bounded for $\theta\to0$, we obtain local energy decay as well as global Strichartz estimates for the solutions of $\partial_t^2u-\Delta_xu+V(t,x)u=F(t,x)$. \end{abstract}

\section*{Introduction}
\renewcommand{\theequation}{\arabic{section}.\arabic{equation}}
\setcounter{equation}{0}
Consider the Cauchy problem
\begin{equation} \label{problem}  \left\{\begin{aligned}
\partial_t^2u-\Delta_{x}u+V(t,x)u&=F(t,x),\ \ t>\tau, \ \ x\in{\R}^{n},\\\
(u,u_{t})(\tau,x)&=(f_{1}(x),f_{2}(x))=f(x),\ \ x\in{\R}^n,\end{aligned}\right.\end{equation}
where $n\geq3$ and the potential $V(t,x)\in\mathcal C^\infty(\R^{1+n},\R)$ satisfies  the conditions
\begin{enumerate}
\item[$\rm(H1)$] there exist $D,\epsilon>0$ such that for all $\alpha\in\mathbb N^n$, $\beta\in\{0,1,2\}$  we have \[\abs{\partial_t^\beta\partial_x^\alpha V(t,x)}\leq C_{\alpha,\beta} e^{-(2D+3\epsilon)\left\langle x\right\rangle},\]
\item[$\rm(H2)$] there exists $T>0$ such that $V(t+T,x)=V(t,x)$.
\end{enumerate} 
Let  $\dot{H}^\gamma({\R}^n)=\Lambda^{-\gamma} (L^2({\R}^n))$ be the homogeneous Sobolev spaces, where $\Lambda=\sqrt{-\Delta_x}$ is determined by the Laplacian in ${\R}^n$. Set $\dot{\mathcal{H}}_\gamma(\R^n)=\dot{H}^\gamma({\R}^n)\times\dot{H}^{\gamma-1}({\R}^n)$ and notice  that for $\gamma<\frac{n}{2}$ and $r>0$ the multiplication with $e^{-r\left\langle x\right\rangle}$ is continuous from $\dot{H}^\gamma({\R}^n)$ to $H^\gamma({\R}^n)$ (this follows from the fact that $e^{-r\left\langle x\right\rangle}\in\mathcal S(\R^n)$). The solution of
\eqref{problem} with $F = 0$ is given by the propagator
\[U(t,\tau):\dot{\mathcal{H}}_\gamma({\R}^n)\ni(f_1,f_2)=f\mapsto U(t,\tau)f=(u,u_t)(t,x)\in\dot{\mathcal{H}}_\gamma({\R}^n)\]
and we refer to \cite{P1}, Chapter V, for the properties of $U(t,\tau)$. Let $U_0(t)$ be the unitary
propagator  in $\dot{\mathcal{H}}_\gamma({\R}^n)$ related to the Cauchy problem \eqref{problem} with $V = 0$, $s = 0$ and let $U(T) = U(T, 0)$.
We have the representation
\[U(t,\tau)=U_0(t-\tau)-\int_\tau^tU(t,s)Q(s)U_0(s)\d s,\quad t\geq\tau\]
where
\[Q(t)=\left(\begin{array}{cc}0&0\\ V(t,x)&0\end{array}\right).\]
By interpolation it is easy to see that 
\begin{equation} \label{energy}\norm{U(t,s)}_{\mathcal L(\dot{\mathcal{H}}_\gamma({\R}^n))}\leq C_\gamma e^{k_\gamma\abs{t-s}},\quad t\geq\tau\end{equation}
where $k_\gamma$ is bounded if $\gamma$ runs in a compact interval.

It is well known that for stationaries potentials ($V(t,x)$ independent of $t$), in order to establish large time estimates of solutions of \eqref{problem}, the main points (see \cite{SZ}, \cite{Va} and \cite{Vo}) are the meromorphic continuation and estimates of the resolvent $(P-\lambda^2)^{-1}$ associate to the hamiltonian $P=-\Delta_x+V$. For time dependent potential, the hamiltonian $P$ is also time-dependent and we can not deduce the large time behavior of solutions of \eqref{problem} from the  properties of the resolvent $(P-\lambda^2)^{-1}$. Moreover, the time dependence of the potential $V(t,x)$ leads to many difficulties (see \cite{CPR} and \cite{P1}). The analysis of the Floquet operator $U(T)$ make it possible to overcome some of these difficulties. Following an idea of \cite{CS}, in \cite{BP1} and \cite{P1} the authors used the Lax-Phillips theory and proved many results including local energy decay and existence of scattering operator. In \cite{P2}, Petkov treats the case of even dimensions by considering the meromorphic continuation of the cut-off resolvent of the Floquet operator $U(T)$ defined by
\[R_{\psi_1,\psi_2}(\theta)=\psi_1(U(T)-e^{-i\theta})^{-1}\psi_2,\quad\psi_1,\psi_2\in\CI.\]
 Note that all these arguments hold only  for potentials $V(t,x)$ which are compactly supported with respect to $x$. Let us introduce, for $B>0$ sufficiently large,  the  resolvent
\[R(\theta)=e^{-D\left\langle x\right\rangle}(U(T)-e^{-i\theta})^{-1}e^{-D\left\langle x\right\rangle}\ :\B\rightarrow\B,\quad \im(\theta)\geq BT\]
with $D$ the constant of (H1).  The purpose of this paper is to extend the works of \cite{BP1}, \cite{BP2}, \cite{CPR}, \cite{P1} and \cite{P2}  to perturbations of the Laplacian which are time-periodic and non-compactly supported  in $x$ by showing the meromorphic continuation of $R(\theta)$.

We recall that the properties of $R(\theta)$ are closely related to the asymptotic expansion of\\
 $e^{-(D+\epsilon)\left\langle x\right\rangle}U(t,0)e^{-D\left\langle x\right\rangle}f$, with $f\in\B$,    as $t\to+\infty$.
Indeed, we can establish (see \cite{Ki4}) the  inversion formula
\begin{equation} \label{inv}e^{-D\left\langle x\right\rangle}U(kT,0)e^{-D\left\langle x\right\rangle}=-\frac{1}{2\pi}\int_{[-\pi+iBT,\pi+iBT]}\hspace{-2cm}e^{-i(k+1)\theta}R(\theta)\d\theta,\quad k\in\mathbb N.\end{equation}
We use the following definition of meromorphic family of bounded operators.
\begin{defi}Let $H_1$ and $H_2$  be two Hilbert spaces. A family of bounded operators
 $Q(\theta):H_1\rightarrow H_2$ is said to be meromorphic in a domain $D\subset\mathbb C$, if $Q(\theta)$ is meromorphically dependent on $\theta$ for $\theta\in D$ and for any pole  $\theta=\theta_0$  the coefficients of the negative powers of  $\theta-\theta_0$ in the appropriate Laurent extension are finite-rank operators. \end{defi}
Now assume that there exists $B_1\in\R$ such that $R(\theta)$ admits a meromorphic continuation with respect to $\theta$ for $\theta\in \{\theta\in\C : \im(\theta)> B_1T\}$. Then, applying \eqref{inv} and integrating $e^{-i(k+1)\theta}R(\theta)$, for $k\in\mathbb N$, on a well chosen contour (see  Lemma 3 of \cite{Ki4} and the proof of the  Main Theorem in \cite{Vai}),  we can establish the asymptotic expansion  of $e^{-(D+\epsilon)\left\langle x\right\rangle}U(t,0)e^{-D\left\langle x\right\rangle}f$  as $t\to+\infty$  with a remainder term whose energy is bounded with respect to $e^{(B_1+\epsilon_1)t}$ for all $\epsilon_1>0$. 

The  goal of this paper is to establish a meromorphic continuation of $R(\theta)$ that allows us to study the asymptotic expansion of $e^{-(D+\epsilon)\left\langle x\right\rangle}U(t,0)e^{-D\left\langle x\right\rangle}f$ as $t\to+\infty$ with a remainder term whose energy decays exponentially  when $n$ is odd and  a remainder term whose energy is bounded with respect to $t^l\log(t)^m$, with $l,m\in\mathbb Z$, when $n$ is even. For this purpose, it suffices to show   that there exists $B_1>0$ such that $R(\theta)$ admits  a meromorphic continuation to $\theta\in \{\theta\in\C : \im(\theta)>- B_1T\}$ for $n$ odd and to 
$\theta\in\{ \theta\in\mathbb C\ :\ \im(\theta)>- B_1T,\ \theta\notin2\pi\mathbb Z+i\R^-\}$  for $n$ even, and, for $n$ even,  to establish the asymptotic expansion as $\theta\to0$ of $R(\theta)$ (see Lemma 3, Lemma 4 in \cite{Ki4} and the proof of the  Main Theorem in \cite{Vai}). The main result of this paper is the following.
\begin{Thm}\label{t3} Assume $\rm(H1)$, $\rm(H2)$  fulfilled and let $n\geq3$. Then, for $D>0$ the constant of $\rm(H1)$,  $R(\theta)$ admits a meromorphic  continuation  to $\{ \theta\in\mathbb C\ :\ \im(\theta)>-\frac{DT}{4}\}$ for $n$ odd and to
\[\{ \theta\in\mathbb C\ :\ \im(\theta)>-\frac{DT}{4},\ \theta\notin2\pi\mathbb Z+i\R^-\}\] for $n$ even. Moreover, for $n$ even, there exists $\epsilon_0>0$ such that for $\theta\in\{ \theta\in\mathbb C\ :\ \abs{\theta}\leq \epsilon_0,\ \theta\notin2\pi\mathbb Z+i\R^-\}$ we have the representation
 \begin{equation}\label{T3}R(\theta)=\sum_{i\geq-m}\sum_{j\geq-m_i}R_{i,j}\theta^i(\log(\theta))^{-j},\end{equation}
where $\log$ is the logarithm defined on $\C\setminus i\R^-$ and, for $i<0$ or $j\neq0$, $R_{i,j}$ are  finite-rank operators.\end{Thm}

Let us recall that in  \cite{P2}, the author established the meromorphic continuation of the cut-off resolvent $R_{\psi_1,\psi_2}(\theta)$ by applying some arguments of \cite{Vai}. Since the analysis of \cite{Vai} is mainly based on the fact that the perturbations are compactly supported in $x$, we can not apply the arguments used in \cite{P2}. Nevertheless, we follow the analysis of \cite{Vai} and we use the transformation $F'$ defined by
\[F'[\phi(t,x)](t,\theta)=e^{\frac{it\theta}{T}}\left(\sum_{k=-\infty}^{+\infty}\phi(t+kT,x)e^{ik\theta}\right),\quad \phi\in\mathcal C^\infty_0(\R^{1+n})\]
which plays the role of the Fourier transform with respect to $t$ for non-stationary and time-periodic problems. Let us remark (see Lemma 2 in \cite{Vai} and Section 1) that an application of $F'$ transforms the solutions of \eqref{problem}, whose energy can grow exponentially in time (see \cite{CPR}), into time-periodic functions.  As in \cite{Vai}, in our analysis, this property  will play a crucial role. 

We prove Theorem \ref{t3} in tree steps. First, in Section 1, we reduce the problem using integrals equations and we show that the meromorphic continuation of $R(\theta)$ follows from an analytic   continuation  of the transformation $F'$ with respect to $\theta$ of the solutions of the free wave equation extended by $0$ for $t<0$. Then, in Section 2, we establish this  analytic continuation. Finally, in Section 3, we conclude by applying the analytic Fredholm theorem and some results of \cite{Vai}. 

Let us observe that the resolvent $R(\theta)$ is defined in homogeneous Sobolev space. In our approach, we need to consider the resolvent
$\left( U(T)-e^{-i\theta}\right)^{-1}$
acting in weighted inhomogeneous Sobolev space. For this purpose, we exploit the fact that, for $r>0$, the multiplication  with $e^{-r\left\langle x\right\rangle}$  maps continuously $\dot{H}^1(\R^n)$ into $H^1(\R^n)$ for $n\geq3$. Notice that for $n\leq 2$ this property does not hold. Therefore, our analysis does not work for $n=2$. To treat the case $n=2$, one can consider the resolvent
\[e^{-D\left\langle x\right\rangle}(U(T)-e^{-i\theta})^{-1}e^{-D\left\langle x\right\rangle}\ :\dot{\mathcal{H}}_\gamma({\R}^n)\to \dot{\mathcal{H}}_\gamma({\R}^n)\]
with $\gamma<1$, and repeat our arguments.

According to  \cite{Ki4} and the proof of the  Main Theorem in \cite{Vai}, from Theorem \ref{t3} we can establish the following results:
\begin{enumerate}
\item[$\rm1.$] Assume $n\geq3$  odd 
and let $\theta_1,\ldots,\theta_N$ be the poles of $R(\theta)$ lying in \[\{\theta\in\C : \im(\theta)\geq -\frac{DT}{4}+\epsilon_1T,\ -\pi\leq\re(\theta)<\pi\}\] with $0<\epsilon_1<\frac{D}{4}$. Then we have the representation
\begin{equation}\label{asym1}e^{-D\left\langle x\right\rangle}U(kT,0)e^{-D\left\langle x\right\rangle}=-i\left[\sum_{j=1}^{N} \underset{\theta=\theta_j}{\textrm{res}}\left(e^{-i(k+1)\theta}R(\theta)\right)\right]+G_k,\quad k\in\mathbb N,\end{equation}
where the remainder term $G_k$ satisfies
\[\norm{G_k}_{\mathcal L(\B)}\lesssim e^{(-\frac{D}{4}+\epsilon_1)kT}.\]
Moreover, from \eqref{asym1}, we deduce the asymptotic expansion of $e^{-(D+\epsilon)\left\langle x\right\rangle}U(t,0)e^{-D\left\langle x\right\rangle}f$  as $t\to+\infty$  with a remainder term whose energy is bounded with respect to $e^{(-\frac{D}{4}+\epsilon_1)t}$.
\item[$\rm2.$] Assume $n\geq4$ even and let $\theta_1,\ldots,\theta_N$ be the poles of $R(\theta)$ lying in \[\{\theta\in\C : \im(\theta)\geq 0,\ -\pi\leq\re(\theta)<\pi,\ \theta\neq0\}.\] Now, remove from \eqref{T3} the terms at which $i\geq0$ and $j=0$ and let $B\theta^\mu(\log(\theta))^\nu$ be the leading term of the remaining term with non-zero coefficients. Then, we obtain the representation
\begin{equation}\label{asym2}e^{-D\left\langle x\right\rangle}U(kT,0)e^{-D\left\langle x\right\rangle}=-i\left[\sum_{j=1}^{N} \underset{\theta=\theta_j}{\textrm{res}}\left(e^{-i(k+1)\theta}R(\theta)\right)\right]+\lambda_k B (1+kT)^{l}(\ln(2+kT))^{m} +G_k,\quad k\in\mathbb N,\end{equation}
where $\lambda_k\in\mathbb C$ is bounded with respect to $k$, $(l,m)=(-\mu-1,\nu)$ for $\mu<0$ and $(l,m)=(-\mu-1,-\nu)$ for $\mu\geq0$, the remainder term $G_k$ satisfies
\[\norm{G_k}_{\mathcal L(\B)}=\underset{k\to+\infty}o\left((1+kT)^{l}(\ln(2+kT))^{m}\right).\]
In addition, applying \eqref{asym1}, we obtain the asymptotic expansion of $e^{-(D+\epsilon)\left\langle x\right\rangle}U(t,0)e^{-D\left\langle x\right\rangle}f$  as $t\to+\infty$  with a remainder term whose energy is bounded with respect to $t^{l}\ln(t)^{m}$.
\end{enumerate}

From the meromorphic continuation of $R(\theta)$, we deduce sufficient conditions for   {\bf local energy decay} and {\bf global Strichartz estimates}. More precisely, we have seen the link between the meromorphic continuation of  $R(\theta)$ and the asymptotic expansion of $e^{-(D+\epsilon)\left\langle x\right\rangle}U(t,0)e^{-D\left\langle x\right\rangle}f$ as $t\to+\infty$. Consequently, it seems natural to consider the meromorphic continuations of  $R(\theta)$ that imply dispersive estimates.  Introduce the following condition.
\begin{enumerate}
\item[$\rm(H3)$] The  resolvent $R(\theta)$ has no poles on  $\{\theta\in\C\ :\ \im(\theta)\geq0\}$ for $n$ odd and on $\{\theta\in\C\ :\ \im(\theta)\geq0,\ \theta\notin2\pi\mathbb Z\}$ for $n$ even. Moreover, for $n$ even we have
\[\limsup_{\substack{\lambda\to0 \\ \im(\lambda)>0}}\Vert R(\lambda)\Vert<\infty.\]
\end{enumerate}
Assuming $\rm(H3)$ fulfilled and applying the arguments  in \cite{Ki3}, \cite{Ki4} and \cite{P2}, in Subsection 4.1 we  establish 
 a local energy decay of the form 
\begin{equation} \label{local}\norm{e^{-(D+\epsilon)\left\langle x\right\rangle}U(t,\tau)e^{-(D+\epsilon)\left\langle x\right\rangle}}_{\mathcal L(\B)}\lesssim p(t-\tau),\quad t\geq \tau
\end{equation}
with $p(t)\in L^1(\R^+)$. More precisely, we obtain the following.
\begin{Thm}\label{t1} Assume $\rm(H1)$, $\rm(H2)$ and $\rm(H3)$ fulfilled and let $n\geq3$. Then, we have \eqref{local} with 
\begin{equation}\label{local1}\left\{\begin{aligned}p(t)&=e^{-\delta t}\textrm{ for $n\geq3$ odd},\\  p(t)&= \frac{1}{(t+1)\ln^2(t+e)}\textrm{ for $n\geq4$ even}.\end{aligned}\right.\end{equation}
\end{Thm}

The decay of local energy for time dependent perturbations has been investigated by many authors (see \cite{BP1}, \cite{CS}, \cite{P2} and \cite{T5}). The main hypothesis is that the perturbations are {\bf non-trapping} (see  \cite{CS}, \cite{Ki1} and \cite{T6}). In contrast to the stationary case (see \cite{Mel}, \cite{Met}, \cite{Va} and \cite{Vo}) the non-trapping condition is not sufficient for a local energy decay. In particular, the problem \eqref{problem}
is non-trapping but we may have solutions with exponentially growing local energy. Indeed, if $R(\theta)$ has a pole $\theta_0\in\C$ with $\im(\theta_0)>0$, from the representations \eqref{asym1} and \eqref{asym2}, one may show that
\[\norm{e^{-(D+\epsilon)\left\langle x\right\rangle}U(t,0)e^{-(D+\epsilon)\left\langle x\right\rangle}}_{\mathcal L(\B)}\gtrsim e^{\frac{\im(\theta_{0})}{T}t},\quad t>0\]
and deduce from this estimate the existence of a solution whose energy grows exponentially.
In \cite{CPR}, the authors established an example of positive potential such that this phenomenon occurs for problem \eqref{problem}. In our case assumption $\rm(H3)$  excludes  the existence of such solutions. Moreover, according to \eqref{asym1} and \eqref{asym2}, $\rm(H3)$ is an optimal condition for a local energy decay \eqref{local} with $p(t)$ satisfying \eqref{local1} (see also Lemma 3 and Lemma 4 of \cite{Ki4}).

Let us recall that the local energy decay is a crucial point (see \cite{Ki1}, \cite{Ki3}, \cite{Met}, \cite{MT} and \cite{P2}) to obtain estimates of the form
\begin{equation}\label{1}\Vert u\Vert_{L_t^p({\R}^+,L_x^q({R}^n))}+ \norm{\left(u, \partial_t(u)\right)}_{L^\infty_t({\R}^+,\dot{\mathcal{H}}_\gamma({\R}^n))}\leq C(p,q,n,\rho,T,\gamma)\left(\Vert f\Vert_{\dot{\mathcal{H}}_\gamma({\R}^n)}+\norm{F}_{L^{\tilde{p}}_t(\R^+,L^{\tilde{q}}_x(\R^n))}\right).\end{equation}
Estimates \eqref{1} are called {\bf global Minkovski Strichartz estimates} and they  are important in order to prove existence and uniqueness
results  for non-linear equations (see for examples \cite{KT}, \cite{LS}, \cite{S1} and \cite{S2}). Following the results of \cite{P2}, in Subsection 4.2 we apply estimate \eqref{local} and prove estimates \eqref{1} for solutions of \eqref{problem}. We obtain the following.
\begin{Thm}\label{t2} Assume $\rm(H1)$, $\rm(H2)$ and $\rm(H3)$ fulfilled and let $n\geq3$. Let  $1\leq\tilde{p},\tilde{q}\leq2\leq p,q\leq\infty$, $0<\gamma\leq1$, $p>2$,  $q<\frac{2(n-1)}{n-3}$ and $\tilde{q}'<\frac{2(n-1)}{n-3}$ satisfy
\begin{equation}\label{2}\frac{1}{p}+\frac{n}{q}=\frac{n}{2}-\gamma=\frac{1}{\tilde{p}}+\frac{n}{\tilde{q}}-2,\quad \frac{1}{p} \leq\frac{(n-1)(q-2)}{4q},\quad \frac{1}{\tilde{p}'} \leq\frac{(n-1)(\tilde{q}'-2)}{4\tilde{q}'}.\end{equation}
Then, the solution $u$ of \eqref{problem} with  $\tau=0$ satisfies \eqref{1}.\end{Thm}
It is well known (see \cite{KT} and Corollary 3.2 of \cite{LS}) that for the solution  of \eqref{problem} with $V=0$ and $\tau=0$,   estimates \eqref{1} hold for $0<\gamma\leq1$  if
 $1\leq\tilde{p},\tilde{q}\leq2\leq p,q\leq\infty$, with  $q<\frac{2(n-1)}{n-3}$ and $\tilde{q}'<\frac{2(n-1)}{n-3}$   satisfy \eqref{2}. Since we apply the Krist-Chiselev lemma in the proof of Theorem \ref{t2} (see \cite{P1}) we must exclude  $p=2$. In fact, the result of Theorem \ref{t2} holds for any values of $p,q,\gamma,\tilde{p},\tilde{q}$ for which the solutions of the free wave equation ($V(t,x)=0$) satisfy \eqref{1} as long as $p>2$ and $0<\gamma\leq1$.
 
 For the assumption $\rm(H3)$, in Subsection 4.3 we give examples of potentials $V(t,x)$ such that $\rm(H1)$, $\rm(H2)$ and $\rm(H3)$ are fulfilled. 
 
 We  like to mention that many authors proved local energy decay as well as global Strichartz estimates for the Schrödinger equation with non-compactly supported in $x$  and time-periodic potentials (see \cite{GJY}, \cite{G}, \cite{Ya} and \cite{Yo}). It seems that our paper is  the first work where one treats local energy decay and global Strichartz estimates for the wave equations with  perturbations depending on $(t,x)$  which are not compactly supported with respect to $x$, without any assumptions on the size of the perturbations (see the work  of \cite{MT} and \cite{MT1} for small perturbations of the Laplacian).

\section{Reduction of the problem}

 For all $\gamma\in\R$, we denote by $\mathcal H_\gamma(\R^n)$ the space
\[\mathcal H_\gamma(\R^n)=H^{\gamma}(\R^n)\times H^{\gamma-1}(\R^n).\] 
Since, for $n\geq3$ and for $0<\delta<D$, the multiplication by $e^{-\delta\left\langle x\right\rangle}$ is a bounded operator from $\B$ to $\Bb$, estimate \eqref{energy} implies that, for all $t>0$, the operator 
\[e^{-\delta\left\langle x\right\rangle}U(t,0)e^{-\delta\left\langle x\right\rangle}:\Bb\to\Bb\]
is bounded. In   this section, as well as in Sections  2 and 3, we assume that
\begin{equation}\label{3} U(t,\tau)=0\quad\textrm{for}\quad t<\tau \quad\textrm{and}\quad U_0(t)=0\quad\textrm{for}\quad t<0.\end{equation}
This assumption  and estimate \eqref{energy} allow  us to consider, for $\im(\theta)\geq BT$, with $B>0$ sufficiently large, and for $0<\delta<D$, the families of bounded operators
\[F'\left[e^{-\delta\left\langle x\right\rangle}U(t,0)e^{-\delta\left\langle x\right\rangle}\right](t,\theta)=e^{\frac{it\theta}{T}}\sum_{k=-\infty}^{+\infty}\left(e^{-\delta\left\langle x\right\rangle}U(t+kT,0)e^{-\delta\left\langle x\right\rangle}e^{ik\theta}\right)\ :\ \Bb\to\Bb\]
and 
\[F'\left[e^{-\delta\left\langle x\right\rangle}U_0(t)e^{-\delta\left\langle x\right\rangle}\right](t,\theta)=e^{\frac{it\theta}{T}}\sum_{k=-\infty}^{+\infty}\left(e^{-\delta\left\langle x\right\rangle}U_0(t+kT)e^{-\delta\left\langle x\right\rangle}e^{ik\theta}\right)\ :\ \Bb\to\Bb.\]
The goal of this section is to show that the meromorphic continuation of $R(\theta)$ described in\\
 Theorem \ref{t3}
follows from an analytic continuation with respect to $\theta$  of
\[F'\left[e^{-\delta\left\langle x\right\rangle}U_0(t)e^{-\delta\left\langle x\right\rangle}\right](t,\theta)\ :\ \Bb\to\Bb,\quad 0<\delta<D,\quad t\in\R.\]
For this purpose, we will use the  Duhamel's principle to establish a link between the propagator $U(t,0)$ of the disturbed wave equation and the propagator $U_0(t)$ of the free wave equation. Then, applying some arguments of Vainberg \cite{Vai}, we will deduce the connection between the transformation $F'$ of both of these propagators. We start by proving that the meromorphic continuation of $R(\theta)$ defined as a family of bounded operator in the homogeneous Sobolev space $\B$ can be deduced from a meromorphic continuation of the resolvent $\left(U(T)-e^{-i\theta}\right)^{-1}$ acting in weighted inhomogeneous Sobolev space.

Let us introduce, for $0<\delta<D$  and for $\im(\theta)>BT$ with $B>0$ sufficiently large, the resolvent
\[e^{-\delta\left\langle x\right\rangle}(U(T)-e^{-i\theta})^{-1}e^{-\delta\left\langle x\right\rangle}.\]
Since, for $n\geq3$, the multiplication by $e^{-\delta\left\langle x\right\rangle}$ is a bounded operator from $\B$ to $\Bb$, the family
\[e^{-\delta\left\langle x\right\rangle}(U(T)-e^{-i\theta})^{-1}e^{-\delta\left\langle x\right\rangle}: \Bb\to \Bb.\]
 is a family of bounded operators mapping $\Bb$ to $\Bb$. 
To prove the meromorphic continuation of $R(\theta)$,
we will use the following result.
\begin{lem}\label{l8} Assume that for all $0<\delta<D$,  the family\[e^{-\delta\left\langle x\right\rangle}(U(T)-e^{-i\theta})^{-1}e^{-\delta\left\langle x\right\rangle}: \Bb\to \Bb\] admits a meromorphic continuation,  with respect to $\theta$,  to $\{ \theta\in\mathbb C\ :\ \im(\theta)>-\frac{\delta T}{4}\}$ for $n$ odd and to $\{ \theta\in\mathbb C\ :\ \im(\theta)>-\frac{\delta T}{4},\ \theta\notin2\pi\mathbb Z+i\R^-\}$ for $n$ even, with an asymptotic expansion as $\theta\to0$ which take the form \eqref{T3}. Then, 
\[R(\theta): \B\to \B\]
admits the meromorphic continuation described in Theorem \ref{t3}.\end{lem}
\begin{proof} Let $0<\delta<D$. Notice that, for $\im(\theta)>BT$, we have
\[R(\theta)=e^{-(D-\delta)\left\langle x\right\rangle}\left(e^{-\delta\left\langle x\right\rangle}(U(T)-e^{-i\theta})^{-1}e^{-\delta\left\langle x\right\rangle}\right)e^{-(D-\delta)\left\langle x\right\rangle}.\]
Since the multiplication by $e^{-(D-\delta)\left\langle x\right\rangle}$ is a bounded operator from $\B$ to $\Bb$, for $\theta\in\{ \theta\in\mathbb C\ :\ \im(\theta)>-\frac{\delta T}{4}\}$ when $n$ is odd and  $\theta\in\{ \theta\in\mathbb C\ :\ \im(\theta)>-\frac{\delta T}{4},\ \theta\notin2\pi\mathbb Z+i\R^-\}$ when $n$ is even, $R(\theta)$ admits the following meromorphic continuation
\[R(\theta)=e^{-(D-\delta)\left\langle x\right\rangle}\left(e^{-\delta\left\langle x\right\rangle}(U(T)-e^{-i\theta})^{-1}e^{-\delta\left\langle x\right\rangle}\right)e^{-(D-\delta)\left\langle x\right\rangle}\ :\B\to\B.\]
Hence, for all $0<\delta<D$,  $R(\theta)$ admits a meromorphic continuation to $\{ \theta\in\mathbb C\ :\ \im(\theta)>-\frac{\delta T}{4}\}$ for $n$  odd and to $\theta\in\{ \theta\in\mathbb C\ :\ \im(\theta)>-\frac{\delta T}{4},\ \theta\notin2\pi\mathbb Z+i\R^-\}$ when $n$ is even, and we deduce\\  Theorem \ref{t3}.\end{proof}
Following Lemma \ref{l8}, we set $0<\delta<D$ and we will establish the meromorphic continuation of \[e^{-\delta\left\langle x\right\rangle}(U(T)-e^{-i\theta})^{-1}e^{-\delta\left\langle x\right\rangle}: \Bb\to \Bb\] 
described in Lemma \ref{l8}.
Let us recall (see \cite{Ki4}) that, for $\im(\theta)\geq BT$ with $B>0$ sufficiently large, we have
\begin{equation}\label{4}e^{-\delta\left\langle x\right\rangle}(U(T)-e^{-i\theta})^{-1}e^{-\delta\left\langle x\right\rangle}=-e^{i\theta}F'\left[e^{-\delta\left\langle x\right\rangle}U(t,0)e^{-\delta\left\langle x\right\rangle}\right](T,\theta).\end{equation}
Thus, for our purpose it  suffices to establish  the meromorphic continuation of the family 
\[F'\left[e^{-\delta\left\langle x\right\rangle}U(t,0)e^{-\delta\left\langle x\right\rangle}\right](t,\theta):\ \Bb\to\Bb\]
with respect to $\theta$. 
We will show that the meromorphic continuation of 
\[F'\left[e^{-\delta\left\langle x\right\rangle}U(t,0)e^{-\delta\left\langle x\right\rangle}\right](t,\theta):\ \Bb\to\Bb\]
follows from an analytic continuation of
\[F'\left[e^{-\delta\left\langle x\right\rangle}U_0(t)e^{-\delta\left\langle x\right\rangle}\right](t,\theta):\ \Bb\to\Bb.\]
Now, let us define two  Hilbert spaces and let us recall some results of  \cite{Vai}.
\begin{defi}
  Let $\gamma\in\R$, $A\in\R$. We denote by $\mathcal H_{\gamma}^A(\R^{1+n})$  the closure of\\
   $\mathcal C_0^\infty(]0,+\infty[\times\R^{n})\times\mathcal C_0^\infty(]0,+\infty[\times\R^{n})$  with respect to the norm $\norm{\cdot}_{\mathcal H_{\gamma}^A}$. Here, $\norm{\cdot}_{\mathcal H_{\gamma}^A}$ is defined by
  \[\norm{\left(\phi_1,\phi_2\right)}^2_{\mathcal H_{\gamma}^A}=\norm{e^{-At}\phi_1}^2_{H^{\gamma}(\R^{1+n})}+\norm{e^{-At}\phi_2}^2_{H^{\gamma-1}(\R^{1+n})}.\]
\end{defi} 
\begin{rem}
To understand the use of the spaces $\mathcal H_{\gamma}^A(\R^{1+n})$, let $h\in\mathcal C^\infty(\R)$ be such that\\
$\textrm{supp}(h)\subset]0,+\infty[$. Then, estimate \eqref{energy} implies that, for $A>k_1$,  $h(t)e^{-\delta\left\langle x\right\rangle}U(t,0)e^{-\delta\left\langle x\right\rangle}$ is a bounded operator from $\Bb$ to $\mathcal H_{1}^A(\R^{1+n})$.\end{rem}

\begin{defi} Let $\gamma\in\R$. We denote by $H^\gamma_{per}(\R^{1+n})$  the set of functions $T$-periodic with respect to $t$ lying in the closure of the set
   \[\{ \phi\in \mathcal C^\infty_t(\R,\CI)\ :\ \phi(t,x)\textrm{ is }T-\textrm{periodic with respect to }t\}\]
  with respect to the norm
  \[\norm{\cdot}_{H^{\gamma}(]0,T[\times\R^n)}.\]
  We set $\mathcal{H}_{\gamma,per}(\R^{1+n})=H^\gamma_{per}(\R^{1+n})\times H^{\gamma-1}_{per}(\R^{1+n})$.\end{defi}
  \begin{rem}
  Let $\theta\in\mathbb C$ and $A>0$ be such that $\im(\theta)>AT$. Then, Vainberg proved in Lemma 2 of \cite{Vai} that the operator
  \begin{equation}\label{6}\begin{array}{rccl} F'_\theta: & \mathcal{H}^{A}_{\gamma}(\R^{1+n})& \to & \mathcal{H}_{\gamma,per}(\R^{1+n}), \\
 \ \\ & \phi(t,.) & \mapsto &  F'(\phi)(t,\theta) \end{array}\end{equation}
 is bounded. This result shows the remarkable property of $F'$. Namely, $F'$ transforms functions, which are  exponentially growing in time, into time-periodic functions.\end{rem}
Applying \eqref{energy}, we deduce easily the following (see also Lemma 1 and Lemma 2 in \cite{Vai}).
\begin{prop}\label{ra} Let $B>k_1$ with $k_1$ the constant of \eqref{energy} when $\gamma=1$. Then, 
for $\im(\theta)>BT$, the operator
\[\begin{array}{rccl} F'\left[e^{-\delta\left\langle x\right\rangle}U(t,0)e^{-\delta\left\langle x\right\rangle}\right](t,\theta): & \mathcal{H}_{1}(\R^{n})& \to & \mathcal{H}_{1,per}(\R^{1+n}), \\
 \ \\ & f & \mapsto &  F'\left[e^{-\delta\left\langle x\right\rangle}U(t,0)e^{-\delta\left\langle x\right\rangle}\right](t,\theta)f\end{array}\]
 is bounded.\end{prop}

Let us recall a result of \cite{Vai}.
\begin{prop}\emph{(Lemma 6,  \cite{Vai})}\label{ro}    For all $r>0$,
the operator
\begin{equation}\label{6}\begin{array}{rccl} W_0: & \mathcal{H}^{r}_{1}(\R^{1+n})& \to & \mathcal{H}^{r}_{1}(\R^{1+n}), \\
 \ \\ & \phi(t,.) & \mapsto &  \int_{-\infty}^tU_0(t-s)\phi(s)\d s. \end{array}\end{equation}
is bounded, and for $\im(\theta)> rT$, we have
\[F'\left[W_0\phi\right](t,\theta)=J(t,\theta)F'(\phi)(t,\theta),\]
where 
\begin{equation}\label{6}\begin{array}{rccl} J(t,\theta): & \mathcal{H}_{1,per}(\R^{1+n})& \to & \mathcal{H}_{1,per}(\R^{1+n}), \\
 \ \\ & \psi(t,.) & \mapsto &  \int_{0}^TF'\left[U_0\right](t-s,\theta)\psi(s)\d s\end{array}\end{equation}
 is a family of bounded operators defined for $\theta\in\C$, $\im(\theta)> rT$.\end{prop}

In fact in  \cite{Vai}, Lemma 6, Vainberg shows this result for 
 $F'\left[W_0^{\alpha_1}\right](t,\theta)$ where $W_0^{\alpha_1}$ is defined by
  \[W_0^{\alpha_1}(\phi)(t)=\int_{-\infty}^t\alpha_1(t-s)U_0(t-s)\phi(s)\d s,\quad \phi\in\mathcal C^\infty_0(\R^{1+n})\]
 with $\alpha_1\in\mathcal C^\infty(\R)$ is such that $\alpha_1(t)=0$ for $t\leq t_0$ and $\alpha_1(t)=1$ for $t\geq t_0+1$. Following the same arguments, we can easily obtain the same result for $W_0^{(1-\alpha_1)}$ and deduce \eqref{6}.

From the Duhamel's principle, we deduce the representation
 \[U(t,0)=U_0(t)-\int_0^tU_0(t-s)Q(s)U(s,0)\d s.\]
 and since $U(t,0)=0$ for $t<0$ we obtain
 \begin{equation}\label{60}U(t,0)=U_0(t)-W_0\left[Q(t)U(t,0)\right](t).\end{equation}
 \begin{lem}\label{loulou} Let $B>k_1$ with $k_1$ the constant of \eqref{energy} when $\gamma=1$. Then, for $\im(\theta)>BT$, we have
 \begin{equation}\label{1222} F'\left[W_0\left[Q(t)U(t,0)e^{-\delta\left\langle x\right\rangle}\right](t)\right](t,\theta)=J(t,\theta)F'\left[Q(t)U(t,0)e^{-\delta\left\langle x\right\rangle}\right](t,\theta).\end{equation}\end{lem}
 \begin{proof}
  Choose $\alpha\in\mathcal C^\infty(\R)$  such that $0\leq \alpha(t)\leq 1$, $\alpha(t)=0$ for $t\leq \frac{T}{5}$ and $\alpha(t)=1$ for $t\geq \frac{T}{4}$.
  Then, assumption (H1) and \eqref{energy} imply
  \[\alpha(t)Q(t)U(t,0)e^{-\delta\left\langle x\right\rangle}\in\mathcal L\left(\Bb,\mathcal{H}^{B}_{1}(\R^{1+n})\right)\]
  and Proposition \ref{ro} yields
  \begin{equation}\label{1223}F'\left[W_0\left[\alpha(t)Q(t)U(t,0)e^{-\delta\left\langle x\right\rangle}\right](t)\right](t,\theta)=J(t,\theta)F'\left[\alpha(t)Q(t)U(t,0)e^{-\delta\left\langle x\right\rangle}\right](t,\theta).\end{equation}
  Moreover, since $(1-\alpha(t))Q(t)U(t,0)e^{-\delta\left\langle x\right\rangle}=0$ for $t\notin[0,T[$, we have
  \[F'\left[(1-\alpha(t))Q(t)U(t,0)e^{-\delta\left\langle x\right\rangle}\right](t,\theta)=e^{\frac{it\theta}{T}}(1-\alpha(t))Q(t)U(t,0)e^{-\delta\left\langle x\right\rangle},\quad t\in[0,T[\]
  and, since $U_0(t)=0$ for $t<0$, we get 
  \[\begin{aligned}W_0\left[(1-\alpha(t))Q(t)U(t,0)e^{-\delta\left\langle x\right\rangle}\right](t)&=\int_{\R}U_0(t-s)(1-\alpha(s))Q(s)U(s,0)e^{-\delta\left\langle x\right\rangle}\d s,\\ &=\int_0^TU_0(t-s)(1-\alpha(s))Q(s)U(s,0)e^{-\delta\left\langle x\right\rangle}\d s.\end{aligned}\]
  Applying $F'$ to this representation , for $\im(\theta)>BT$, we obtain
 \[\begin{aligned}&F'\left[\int_0^TU_0(t-s)(1-\alpha(s))Q(s)U(s,0)e^{-\delta\left\langle x\right\rangle}\d s\right](t,\theta)\\ &=\int_0^TF'\left[U_0\right](t-s,\theta)e^{\frac{is\theta}{T}}(1-\alpha(s))Q(s)U(s,0)e^{-\delta\left\langle x\right\rangle}\d s,\\
 &=J(t,\theta)F'\left[(1-\alpha(t))Q(t)U(t,0)e^{-\delta\left\langle x\right\rangle}\right](t,\theta).\end{aligned}.\]
  Combining this equality with \eqref{1223}, we obtain \eqref{1222}.\end{proof}
  
Following \eqref{1222}, applying $F'$ to both sides of \eqref{60} (see also \cite{Ki3} and \cite{Vai}  for the properties of the transformation $F'$) and multiplying by $e^{-\delta\left\langle x\right\rangle}$ on the right and on the left of \eqref{60}  for $\im(\theta)\geq BT$ with $B>0$ sufficiently large, we obtain
 \begin{equation}\label{5}\begin{aligned}F'\left[e^{-\delta\left\langle x\right\rangle}U(t,0)e^{-\delta\left\langle x\right\rangle}\right](t,\theta)=&F'\left[e^{-\delta\left\langle x\right\rangle}U_0(t)e^{-\delta\left\langle x\right\rangle}\right](t,\theta)\\ \ &-e^{-\delta\left\langle x\right\rangle}J(t,\theta)F'\left[Q(t)U(t,0)e^{-\delta\left\langle x\right\rangle}\right](t,\theta).\end{aligned}\end{equation}
  Notice that, for $\im(\theta)>BT$, we have 
 \[F'\left[Q(t)U(t,0)e^{-\delta\left\langle x\right\rangle}\right](t,\theta)=e^{\frac{it\theta}{T}}\left(\sum_{k=-\infty}^{+\infty}Q(t+kT)U(t+kT,0)e^{-\delta\left\langle x\right\rangle}e^{ik\theta}\right).\]
Since $Q(t)$ is $T$-periodic and, from (H1), $Q(t)e^{\delta\left\langle x\right\rangle}\in\mathcal L(\Bb)$, we obtain
 \[\begin{aligned}F'\left[Q(t)U(t,0)e^{-\delta\left\langle x\right\rangle}\right](t,\theta)&=Q(t)e^{\delta\left\langle x\right\rangle}e^{\frac{it\theta}{T}}\left(\sum_{k=-\infty}^{+\infty}e^{-\delta\left\langle x\right\rangle}U(t+kT,0)e^{-\delta\left\langle x\right\rangle}e^{ik\theta}\right)\\ \ &=Q(t)e^{\delta\left\langle x\right\rangle}F'\left[e^{-\delta\left\langle x\right\rangle}U(t,0)e^{-\delta\left\langle x\right\rangle}\right](t,\theta).\end{aligned}\]
 Applying this formula to \eqref{5}, we get
 \[\begin{aligned}F'\left[e^{-\delta\left\langle x\right\rangle}U(t,0)e^{-\delta\left\langle x\right\rangle}\right](t,\theta)=&F'\left[e^{-\delta\left\langle x\right\rangle}U_0(t)e^{-\delta\left\langle x\right\rangle}\right](t,\theta)\\ \ &-e^{-\delta\left\langle x\right\rangle}J(t,\theta)Q(t)e^{\delta\left\langle x\right\rangle}F'\left[e^{-\delta\left\langle x\right\rangle}U(t,0)e^{-\delta\left\langle x\right\rangle}\right](t,\theta)\end{aligned} \]
 and it follows
 \[\left(\textrm{Id}+e^{-\delta\left\langle x\right\rangle}J(t,\theta)Q(t)e^{\delta\left\langle x\right\rangle}\right)F'\left[e^{-\delta\left\langle x\right\rangle}U(t,0)e^{-\delta\left\langle x\right\rangle}\right](t,\theta)=F'\left[e^{-\delta\left\langle x\right\rangle}U_0(t)e^{-\delta\left\langle x\right\rangle}\right](t,\theta).\]
Thus, to prove the meromorphic continuation of $R(\theta)$, we only need to show that the families

\[\left(\textrm{Id}+e^{-\delta\left\langle x\right\rangle}J(t,\theta)Q(t)e^{\delta\left\langle x\right\rangle}\right)^{-1},\quad F'\left[e^{-\delta\left\langle x\right\rangle}U_0(t)e^{-\delta\left\langle x\right\rangle}\right](t,\theta)\]
admit a meromorphic continuation with respect to $\theta$.  
\begin{prop}\label{p1}
Let $G$ be the operator defined by
\[\begin{array}{rccl} G: & \mathcal{H}_{1,per}(\R^{1+n})& \to & \mathcal{H}_{1,per}(\R^{1+n}), \\
 \ \\ & \phi(t,.) & \mapsto &  e^{\delta\left\langle x\right\rangle}Q(t)e^{\delta\left\langle x\right\rangle}\phi(t,.).\end{array}\]
Then $G$ is a compact operator.\end{prop}
\begin{proof}
Let $h=(h_1,h_2)\in \mathcal{H}_{1,per}(\R^{1+n})$. We have
\[e^{\delta\left\langle x\right\rangle}Q(t)e^{\delta\left\langle x\right\rangle}h=(0,e^{2\delta\left\langle x\right\rangle}Vh_1).\]
Applying (H1), we show that $e^{2\delta\left\langle x\right\rangle}Vh_1$ takes value in
$e^{-\epsilon\left\langle x\right\rangle}H^1_{per}(\R^{1+n})$ and the Rellich-Kondrachov theorem for unbounded domains implies that $G$ is a compact operator on $\mathcal{H}_{1,per}(\R^{1+n})$.\end{proof}

Let us remark  that for $\im(\theta)\geq BT$, we have
\[e^{-\delta\left\langle x\right\rangle}J(t,\theta)Q(t)e^{\delta\left\langle x\right\rangle}=e^{-\delta\left\langle x\right\rangle}J(t,\theta)e^{-\delta\left\langle x\right\rangle}e^{\delta\left\langle x\right\rangle}Q(t)e^{\delta\left\langle x\right\rangle}\]
with
\[e^{-\delta\left\langle x\right\rangle}J(t,\theta)e^{-\delta\left\langle x\right\rangle}\phi=\int_0^{T}F'\left[e^{-\delta\left\langle x\right\rangle}U_0(t)e^{-\delta\left\langle x\right\rangle}\right](t-s,\theta)\phi(s)\d s.\]
Combining this identity, Proposition \ref{p1} and representation \eqref{6} with the Fredholm theorem  and Theorem 8 in \cite{Vai}, we conclude that to prove Theorem \ref{t3} it remains  to show  that  $F'\left[e^{-\delta\left\langle x\right\rangle}U_0(t)e^{-\delta\left\langle x\right\rangle}\right](t,\theta)$ admits an analytic continuation in $\theta$ continuous  with respect to $t$, and that  there exists $\theta_0\in\C$ with $\im(\theta_0)>0$ such that the operator
\[\textrm{Id}+e^{-\delta\left\langle x\right\rangle}J(t,\theta_0)Q(t)e^{\delta\left\langle x\right\rangle}\]
is invertible.
\section{Analytic continuation of $F'\left[e^{-\delta\left\langle x\right\rangle}U_0(t)e^{-\delta\left\langle x\right\rangle}\right](t,\theta)$ }
Since $U_0(t)$ is unitary in $\B$ and the multiplication by $e^{-\delta\left\langle x\right\rangle}$ is a continuous operator from $\B$ to $\Bb$, we can easily prove (see Lemma 2 in \cite{Vai}) that 
\[F'\left[e^{-\delta\left\langle x\right\rangle}U_0(t)e^{-\delta\left\langle x\right\rangle}\right](t,\theta):\Bb\to\Bb\]
is well defined and analytic with respect to $\theta$ in $\{\theta\in\C\  :\ \im(\theta)>0\}$. The goal of   this section  is to establish the following.
\begin{Thm}\label{t4} The family of operators
 \[F'\left[e^{-\delta\left\langle x\right\rangle}U_0(t)e^{-\delta\left\langle x\right\rangle}\right](t,\theta):\Bb\rightarrow\Bb\]
  admits an analytic continuation with respect to $\theta$, which is continuous with respect to $t\in\R$,  from 
 $\{\theta\in\C\ :\ \im(\theta)>0\}$ to $\{ \theta\in\mathbb C\ :\ \im(\theta)>-\frac{\delta T}{4}\}$ for $n$ odd and to
  $\{ \theta\in\mathbb C\ :\ \im(\theta)>-\frac{\delta T}{4},\ \theta\notin2\pi\mathbb Z+i\R^-\}$
   for $n$ even. Moreover, for $n$ even, there exists $\epsilon_0>0$ such that for $\theta\in\{ \theta\in\mathbb C\ :\ \abs{\theta}\leq\epsilon_0,\ \theta\notin2\pi\mathbb Z+i\R^-\}$ we have the representation
 \[F'\left[e^{-\delta\left\langle x\right\rangle}U_0(t)e^{-\delta\left\langle x\right\rangle}\right](t,\theta)=B(\theta)\theta^{n-2}\log(\theta)+C(t,\theta),\]
 where $\log$ is the logarithm defined on $\C\setminus i\R^-$,  $C(t,\theta)$ is $\mathcal C^\infty$ and $T$-periodic with respect to $t$ and analytic with respect to $\theta$, $B(\theta)$ is analytic with respect to $\theta$ for  $\abs{\theta}\leq \epsilon_0$, and  for all  $j\in\mathbb N$,  $\left(\partial_\theta^jB(\theta)\right)_{|\theta=0}$ are finite rank operators.\end{Thm}

 Let $\left(\phi_m\right)_{m\in\mathbb N}\subset\mathcal C^\infty(\R^n)$ be a partition of unity such that 
 \[\sum_{m=0}^{+\infty}\phi_m(x)=1,\]
 \[\norm{\partial_x^{\alpha}\phi_m}_{L^\infty(\R^n)}\leq C_\alpha,\quad 0\leq\phi_m\leq1,\quad  m\in\mathbb N,\quad\alpha\in\mathbb N^n,\]
 \[\textrm{supp}\phi_m\subset\{x\in\R^n\ :\ 2^{m-1}T\leq\abs{x}\leq2^{m+1}T\},\quad m\geq1,\]
 \[\textrm{supp}\phi_0\subset\{x\in\R^n\ :\ \abs{x}\leq2T\}.\]
 The main point of Theorem \ref{t4} is to consider, for $\im(\theta)>0$, the following representation
 \begin{equation}\label{7}F'\left[e^{-\delta\left\langle x\right\rangle}U_0(t)e^{-\delta\left\langle x\right\rangle}\right](t,\theta)=\sum_{m,l=0}^{+\infty}e^{-\delta\left\langle x\right\rangle}F'\left[\phi_mU_0(t)\phi_l\right](t,\theta)e^{-\delta\left\langle x\right\rangle}.\end{equation}
Using this representation we will obtain  Theorem \ref{t4}, by proving the analytic continuation of \[e^{-\delta\left\langle x\right\rangle}F'\left[\phi_mU_0(t)\phi_l\right](t,\theta)e^{-\delta\left\langle x\right\rangle}: \Bb\to\Bb\] and  showing that for all $\theta\in\{ \theta\in\mathbb C\ :\ \im(\theta)>-\frac{\delta T}{4}\}$ for $n$ odd (respectively\\ $\theta\in\{ \theta\in\mathbb C\ :\ \im(\theta)>-\frac{\delta T}{4},\ \theta\notin2\pi\mathbb Z+i\R^-\}$ for $n$ even)
\[\sum_{m,l\in\mathbb N}e^{-\delta\left\langle x\right\rangle}F'\left[\phi_mU_0(t)\phi_l\right](t,\theta)e^{-\delta\left\langle x\right\rangle}\]
converge to $\mathcal G(t,\theta)$ which satisfies the properties described in Theorem \ref{t4}.  We treat separately, the case of odd and even dimensions.

\subsection{Analytic continuation of $F'\left[e^{-\delta\left\langle x\right\rangle}U_0(t)e^{-\delta\left\langle x\right\rangle}\right](t,\theta)$ for $n$ odd}
In this subsection we treat the case of  odd dimensions. For $n$ odd, the Huygens principle implies that, for all $m,l\in\mathbb N$, the infinite sum  \[\sum_{k=-\infty}^{+\infty}e^{-\delta\left\langle x\right\rangle}\phi_mU_0(t+kT)e^{ik\theta}\phi_le^{-\delta\left\langle x\right\rangle}\] 
is in fact a finite sum. Thus, $e^{-\delta\left\langle x\right\rangle}F'\left[\phi_mU_0(t)\phi_l\right](t,\theta)e^{-\delta\left\langle x\right\rangle}$ admits 
 an analytic continuation on $\C$ and we can estimate it to obtain the following.

\begin{lem}\label{l1} Let $n\geq3$ be odd. Then, the family of operators
 \[F'\left[e^{-\delta\left\langle x\right\rangle}U_0(t)e^{-\delta\left\langle x\right\rangle}\right](t,\theta):\Bb\rightarrow\Bb\]
  admits an analytic continuation with respect to $\theta$, which is continuous with respect to $t\in\R$,  from 
 $\{\theta\in\C\ :\ \im(\theta)>0\}$ to $\{ \theta\in\mathbb C\ :\ \im(\theta)>-\frac{\delta T}{4}\}$.\end{lem}
 \begin{proof} Since $e^{-\delta\left\langle x\right\rangle}F'\left[\phi_mU_0(t)\phi_l\right](t,\theta)e^{-\delta\left\langle x\right\rangle}$ is $T$-periodic (see \cite{Ki3} and \cite{Vai}) we will only need to estimate it for $0\leq t\leq T$. Let $0\leq t\leq T$. Since, for all $m,l\in\mathbb N$, we have \[\abs{x}+\abs{x_0}\leq2^{m+1}T+2^{l+1}T,\quad  x\in\textrm{supp}\phi_m,\  x_0\in\textrm{supp}\phi_l, \] an application of the Huygens's principle yields
 \[\phi_mU_0(t+kT)\phi_l=0,\quad\textrm{for } k\geq2^{m+1}+2^{k+1}+1.\]
 Thus, for all $m,l\in\mathbb N$, $F'\left[\phi_mU_0(t)\phi_l\right](t,\theta)$ admits an analytic continuation to $\C$ and we have
 \[F'\left[\phi_mU_0(t)\phi_l\right](t,\theta)=e^{i\frac{t\theta}{T}}\left(\sum_{k=0}^{2^{m+1}+2^{l+1}}\phi_mU_0(t+kT)\phi_le^{ik\theta}\right).\]
Since for $m\geq1$ $\textrm{supp}\phi_m\subset\{ x\ :\ \abs{x}\geq2^{m-1}T\}$, one may show that, for all $0<r<\delta$, for $ m,l\geq1$,  we have
 \begin{equation}\label{8}\norm{e^{-\delta\left\langle x\right\rangle}F'\left[\phi_mU_0(t)\phi_l\right](t,\theta)e^{-\delta\left\langle x\right\rangle}}\lesssim e^{-(\delta-r)2^{m-1}T}e^{-(\delta-r)2^{l-1}T}\norm{F'\left[e^{-r\left\langle x\right\rangle}\phi_mU_0(t)\phi_le^{-r\left\langle x\right\rangle}\right](t,\theta)},\ \end{equation}
 and  get
\begin{equation}\label{9}\norm{e^{-\delta\left\langle x\right\rangle}F'\left[\phi_mU_0(t)\phi_0\right](t,\theta)e^{-\delta\left\langle x\right\rangle}}\lesssim e^{-(\delta-r)2^{m-1}T}\norm{F'\left[e^{-r\left\langle x\right\rangle}\phi_mU_0(t)\phi_0e^{-r\left\langle x\right\rangle}\right](t,\theta)},\  m\geq1,\end{equation} 
\begin{equation}\label{10}\norm{e^{-\delta\left\langle x\right\rangle}F'\left[\phi_0U_0(t)\phi_l\right](t,\theta)e^{-\delta\left\langle x\right\rangle}}\lesssim e^{-(\delta-r)2^{l-1}T}\norm{F'\left[e^{-r\left\langle x\right\rangle}\phi_0U_0(t)\phi_le^{-r\left\langle x\right\rangle}\right](t,\theta)},\  l\geq1.\end{equation}
For this purpose, we can choose $\psi_m\in\CI$ such that $\psi_m=1$ on supp$\phi_m$ supported on\\
 $\{ x\ :\ 2^{m-1}T-\frac{T}{4}\leq\abs{x}\leq2^{m+1}T+\frac{T}{4}\}$ with derivatives bounded independently of $m$ and notice that
\[\phi_me^{-\delta\left\langle x\right\rangle}=e^{-\delta\left\langle x\right\rangle}\phi_m=\left(e^{-(\delta-r)\left\langle x\right\rangle}\psi_m\right)e^{-r\left\langle x\right\rangle}\phi_m=\phi_me^{-r\left\langle x\right\rangle}\left(e^{-(\delta-r)\left\langle x\right\rangle}\psi_m\right)\]
with
\[\norm{e^{-(\delta-r)\left\langle x\right\rangle}\psi_mh}_{\Bb}\lesssim e^{-(\delta-r)2^{m-1}T}\norm{h}_{\Bb},\quad h\in\Bb.\]

Let $K\subset\{ \theta\in\mathbb C\ :\ \im(\theta)>-\frac{\delta T}{4}\}$  be a compact set.
According to representation \eqref{7} and estimates \eqref{8}, \eqref{9} and \eqref{10}, it remains to show that there exists $0<r<\delta$ such that we have
\begin{equation}\label{11}\begin{aligned}&\sum_{m,l=1}^{+\infty}\sup_{(t,\theta)\in [0,T]\times K}e^{-(\delta-r)2^{m-1}T}e^{-(\delta-r)2^{l-1}T}\norm{F'\left[e^{-r\left\langle x\right\rangle}\phi_mU_0(t)\phi_le^{-r\left\langle x\right\rangle}\right](t,\theta)}_{\mathcal L(\Bb)}\\
&+\sum_{m=0}^{+\infty}\sup_{(t,\theta)\in [0,T]\times K}e^{-(\delta-r)2^{m-1}T}\norm{F'\left[e^{-r\left\langle x\right\rangle}\phi_mU_0(t)\phi_0e^{-r\left\langle x\right\rangle}\right](t,\theta)}_{\mathcal L(\Bb)}\\
&+\sum_{l=1}^{+\infty}\sup_{(t,\theta)\in [0,T]\times K}e^{-(\delta-r)2^{l-1}T}\norm{F'\left[e^{-r\left\langle x\right\rangle}\phi_0U_0(t)\phi_le^{-r\left\langle x\right\rangle}\right](t,\theta)}_{\mathcal L(\Bb)}<+\infty\end{aligned}\end{equation}
Since $K$ is a compact set, there exists $\delta_1>0$, such that, for all $\theta\in K$, $\im(\theta)\geq\delta_1-\frac{\delta T}{4}$. Then, for all $\theta\in K$, $t\in\R$, $0<r<\delta$ and $m,l\in\mathbb N$, we obtain
\[ \norm{F'\left[e^{-r\left\langle x\right\rangle}\phi_mU_0(t)\phi_le^{-r\left\langle x\right\rangle}\right](t,\theta)}_{\mathcal L(\Bb)}\lesssim \sum_{k=0}^{2^{m+1}+2^{l+1}}\norm{e^{-r\left\langle x\right\rangle}\phi_mU_0(t+kT)\phi_le^{-r\left\langle x\right\rangle}}e^{\left(\frac{\delta T}{4}-\delta_1\right)k}.\]
Since the multiplication by $e^{-r\left\langle x\right\rangle}$ is continuous from $\B$ to $\Bb$, for all $k\in\mathbb N$, we obtain
\[\begin{aligned}\norm{e^{-r\left\langle x\right\rangle}\phi_mU_0(t+kT)\phi_le^{-r\left\langle x\right\rangle}}_{\mathcal L(\Bb)}&=\norm{\phi_me^{-r\left\langle x\right\rangle}U_0(t+kT)e^{-r\left\langle x\right\rangle}\phi_l}_{\mathcal L(\Bb)}\\
&\lesssim\norm{e^{-r\left\langle x\right\rangle}U_0(t+kT)e^{-r\left\langle x\right\rangle}}_{\mathcal L(\Bb)}\\
&\lesssim\norm{U_0(t+kT)}_{\mathcal L(\B)}\\
&\lesssim1.\end{aligned}\]
Thus, we get
\[ \norm{F'\left[e^{-r\left\langle x\right\rangle}\phi_mU_0(t)\phi_le^{-r\left\langle x\right\rangle}\right](t,\theta)}_{\mathcal L(\Bb)}\lesssim\sum_{k=0}^{2^{m+1}+2^{l+1}}e^{\left(\frac{\delta T}{4}-\delta_1\right)k}\]
and it follows
\[ \norm{F'\left[e^{-r\left\langle x\right\rangle}\phi_mU_0(t)\phi_le^{-r\left\langle x\right\rangle}\right](t,\theta)}_{\mathcal L(\Bb)}\lesssim\max\left(e^{\left(\frac{\delta T}{4}-\delta_1\right)\left(2^{m+1}+2^{l+1}\right)},2^{m+1}+2^{l+1}+1\right) .\]
Combining this inequality with estimates \eqref{8}, \eqref{9} and \eqref{10} for $r<\min\left(\delta,\delta_1\right)$, we obtain \eqref{11}. This completes the proof.\end{proof}

\subsection{Analytic continuation of $F'\left[e^{-\delta\left\langle x\right\rangle}U_0(t)e^{-\delta\left\langle x\right\rangle}\right](t,\theta)$ for $n$ even}
In this subsection we treat the case of even dimensions. Since for $n$ even the Huygens principle does not hold,  we must use  other arguments to prove Theorem \ref{t4}. 
Let $h_{m,l}\in\mathcal C^\infty(\R)$ be such that 
 \[0\leq h_{m,l}(t)\leq 1,\quad \sup_{m,l\in\mathbb N}\norm{\partial_t^jh_{m,l}}_{L^\infty(\R)}\leq C_j<\infty,\]
 \[h_{m,l}(t)=1\quad \textrm{for}\quad t\geq 2^{m+1}T+2^{l+1}T+\frac{T}{2}\]
 \[h_{m,l}(t)=0\quad \textrm{for}\quad t\leq 2^{m+1}T+2^{l+1}T+\frac{T}{3}.\]  
Then for $\im(\theta)>0$ and $m,l\in\mathbb N$, we obtain
\[F'\left[\phi_mU_0(t)\phi_l\right](t,\theta)=F'\left[(1-h_{m,l}(t))\phi_mU_0(t)\phi_l\right](t,\theta)+F'\left[h_{m,l}(t)\phi_mU_0(t)\phi_l\right](t,\theta).\]
Since $1-h_{m,l}(t)=0$ for $t\geq 2^{m+1}T+2^{l+1}T+\frac{T}{2}$, repeating the arguments used in the proof of Lemma \ref{l1}, we obtain the following.
\begin{lem}\label{l9}
Let $n\geq4$ be even. Then, for $\theta\in\{ \theta\in\mathbb C\ :\ \im(\theta)>-\frac{\delta T}{4}\}$, 
\[\sum_{m,l\in\mathbb N}F'\left[e^{-\delta\left\langle x\right\rangle}(1-h_{m,l}(t))\phi_mU_0(t)\phi_le^{-\delta\left\langle x\right\rangle}\right](t,\theta):\ \Bb\to\Bb\]
converge to a family of operator  analytic with respect to $\theta$, continuous with respect to $t$.\end{lem}
Following Lemma \ref{l9}, if we show that
\begin{equation}\label{cou}\sum_{m,l\in\mathbb N}F'\left[e^{-\delta\left\langle x\right\rangle}h_{m,l}(t)\phi_mU_0(t)\phi_le^{-\delta\left\langle x\right\rangle}\right](t,\theta):\ \Bb\to\Bb\end{equation}
converge to a family of operators  analytic with respect to $\theta$ satisfying the properties described in Theorem \ref{t4}, we will deduce the analytic continuation 
of $F'\left[e^{-\delta\left\langle x\right\rangle}U_0(t)e^{-\delta\left\langle x\right\rangle}\right](t,\theta)$ for $n$ even. 

The main point in the proof of  Theorem \ref{t4} for $n$ even is  the convergence of \eqref{cou}. In order to prove this result  we will apply the Herglotz-Petrovskii formula (see \cite{GS}, Chapter X of \cite{Va} and the proof of Lemma 6 in \cite{Vai})  that gives  an explicit formula of the  Green's function associate to some solutions of the free wave equation and replaces the Huygens principle which is lacking when the dimension is even. 

Let us recall the Herglotz-Petrovskii formula for the free wave equation.
Consider $E^0(t,x,x_0)$ the Green's function  of \eqref{problem} with $n\geq2$ even, $V=0$, $\tau=0$, $F=0$ and $f_1=0$. More precisely, $E^0(t,x,x_0)$ is the solution of 
\begin{equation}   \left\{\begin{aligned}
(\partial_t^2-\Delta_{x})E^0&=0,\ \ t>0,\ \ x\in\R^n,\\
E^0(0,x,x_0)&=0,\\
\partial_tE^0(0,x,x_0)&=\delta(x-x_0),\end{aligned}\right.\end{equation}
where $\delta(x-x_0)$ is the delta-function. For all $a,b,c>0$  and all  domains \[Q_{a,b,c}=\{(t,x,x_0)\ :\ t\geq a+b+c,\ \abs{x}\leq a,\ \abs{x_0}\leq b\},\]
by the Herglotz-Petrovskii formula, we have
\begin{equation}\label{hp}E^0(t,x,x_0)=\int_{\mathbb S^{n-1}}\frac{\d\omega(\sigma)}{(\left\langle x-x_0,\sigma\right\rangle+t)^{n-1}},\quad (t,x,x_0)\in Q_{a,b,c}.\end{equation}

Now, consider the operator 
\[\frac{\sin(t\Lambda)}{\Lambda} : L^2(\R^n)\rightarrow \dot{H}^1(\R^n)\]
with $\Lambda=\sqrt{-\Delta}$. The solution $u$ of \eqref{problem} with $V=0$, $\tau=0$, $F=0$ and $f_1=0$ is given by \[u(t)=\frac{\sin(t\Lambda)}{\Lambda}f_2.\]
 Applying \eqref{hp}, for all $m,l\in\mathbb N$, we obtain an explicit formula of the kernel of $h_{m,l}(t)\phi_m\frac{\sin(t\Lambda)}{\Lambda}\phi_l$ (see the proof of Lemma \ref{l2}). Using this formula, we deduce the analytic continuation  with respect to $\theta$ of the family 
\[F'\left[e^{-\delta\left\langle x\right\rangle}h_{m,l}(t)\phi_m\frac{\sin(t\Lambda)}{\Lambda}\phi_le^{-\delta\left\langle x\right\rangle}\right](t,\theta).\]
 Then, in Lemma \ref{l2}, we establish the analytic continuation with respect to $\theta$ of the family  
\[\sum_{m,l=0}^{+\infty}F'\left[e^{-\delta\left\langle x\right\rangle}h_{m,l}(t)\phi_m\frac{\sin(t\Lambda)}{\Lambda}\phi_le^{-\delta\left\langle x\right\rangle}\right](t,\theta).\]
Combining the result of Lemma \ref{l2} with the Duhamel's principle we show the convergence of \eqref{cou}
and we deduce Theorem \ref{t4} for $n$ even.

\begin{lem}\label{l2}
Let $n\geq4$ be even. Then 
\[\sum_{m,l\in\mathbb N}F'\left[e^{-\delta\left\langle x\right\rangle}h_{m,l}(t)\phi_m\frac{\sin(t\Lambda)}{\Lambda}\phi_le^{-\delta\left\langle x\right\rangle}\right](t,\theta):L^{2}(\R^n)\rightarrow H^1(\R^n)\]
converge to a family of operators analytic  with respect to $\theta$, continuous with respect to $t\in\R$ and   satisfying the properties described in Theorem \ref{t4}. Moreover,
\[\sum_{m,l\in\mathbb N}F'\left[e^{-\delta\left\langle x\right\rangle}h_{m,l}(t)\phi_m\frac{\sin(t\Lambda)}{\Lambda}\phi_le^{-\delta\left\langle x\right\rangle}\right](t,\theta):L^{2}(\R^n)\rightarrow L^2(\R^n)\]
converge to a family of operators analytic  with respect to $\theta$, $\mathcal C^1$ with respect to $t\in\R$ and  satisfying the properties described in Theorem \ref{t4}.\end{lem}
\begin{proof} 
Let $E_{m,l}(t,x,x_0)$ be the kernel of $e^{-\delta\left\langle x\right\rangle}h_{m,l}(t)\phi_m\frac{\sin(t\Lambda)}{\Lambda}\phi_le^{-\delta\left\langle x\right\rangle}$.
Since \[\abs{x}+\abs{x_0}\leq 2^{m+1}T+2^{l+1}T,\quad  x\in\textrm{supp}\phi_m,\ x_0\in\textrm{supp}\phi_l\] and $h_{m,l}(t)=0$ for $t\leq 2^{m+1}T+2^{l+1}T+\frac{T}{3}$, the formula \eqref{hp} implies 
\[E_{m,l}(t,x,x_0)=\int_{\mathbb S^{n-1}}\frac{e^{-\delta\left\langle x\right\rangle}h_{m,l}(t)\phi_m(x)\phi_l(x_0)e^{-\delta\left\langle x_0\right\rangle}}{(\left\langle x-x_0,\sigma\right\rangle+t)^{n-1}}\d\omega(\sigma).\]
Now notice that  for $\im(\theta)>0$, we have the representation
\[F'\left[E_{m,l}(t,x,x_0)\right](t,\theta)=\int_{\mathbb S^{n-1}}e^{-\frac{i\theta\left\langle x-x_0,\sigma\right\rangle}{T}}v_{m,l}(t,x,x_0,\sigma,\theta)\d\omega(\sigma)\]
with
\[v_{m,l}(t,x,x_0,\sigma,\theta)=\sum_{k=0}^{+\infty}\left(\frac{e^{-\delta\left\langle x\right\rangle}h_{m,l}(t+kT)\phi_m(x)\phi_l(x_0)e^{-\delta\left\langle x_0\right\rangle}}{(\left\langle x-x_0,\sigma\right\rangle+t+kT)^{n-1}}\right)e^{\frac{i\theta\left(\left\langle x-x_0,\sigma\right\rangle+kT+t\right)}{T}}.\]
Then, for $0\leq t<T$, we get
\[\left(-iT\partial_\theta\right)^{n-1}v_{m,l}(t,x,x_0,\sigma,\theta)=\left(\sum_{k=0}^{+\infty}h_{m,l}(t+kT)e^{\frac{i\theta\left(\left\langle x-x_0,\sigma\right\rangle+kT+t\right)}{T}}\right)e^{-\delta\left\langle x\right\rangle}\phi_m(x)\phi_l(x_0)e^{-\delta\left\langle x_0\right\rangle},\]

\begin{equation}\label{lulu}\begin{array}{l}\left(-iT\partial_\theta\right)^{n-1}v_{m,l}(t,x,x_0,\sigma,\theta)\\
=\left(h_{m,l}(t+2^{m+1}T+2^{l+1}T)e^{i\left(2^{m+1}+2^{l+1}\right)\theta}+\frac{e^{i\left(2^{m+1}+2^{l+1}+1\right)\theta}}{\left(1-e^{i\theta}\right)}\right)e^{\frac{i\theta\left(\left\langle x-x_0,\sigma\right\rangle+t\right)}{T}}e^{-\delta\left\langle x\right\rangle}\phi_m(x)\phi_l(x_0)e^{-\delta\left\langle x_0\right\rangle}.\end{array}\end{equation}
Let $E(t,x,x_0)$ be the kernel of
\[\sum_{m,l=0}^{+\infty}h_{m,l}(t)\phi_m\frac{\sin(t\Lambda)}{\Lambda}\phi_l.\]
For $\im(\theta)>0$, we have
\begin{equation}\label{13}F'\left[E(t,x,x_0)\right](t,\theta)=\sum_{m,l=0}^{+\infty}F'\left[E_{m,l}(t,x,x_0)\right](t,\theta)=\int_{\mathbb S^{n-1}}e^{-\frac{i\theta\left\langle x-x_0,\sigma\right\rangle}{T}}v(t,x,x_0,\sigma,\theta)\d\omega(\sigma),\end{equation}
where
\[v(t,x,x_0,\sigma,\theta)=\sum_{m,l=0}^{+\infty}v_{m,l}(t,x,x_0,\sigma,\theta).\]
Repeating the arguments used in the proof Lemma \ref{l1}, we deduce that
\[\sum_{m,l=0}^{+\infty}(1-e^{i\theta})\left(-iT\partial_\theta\right)^{n-1}v_{m,l}(t,x,x_0,\sigma,\theta)\]
admits an analytic continuation to $\{\theta\in\mathbb C\ :\ \im(\theta)>-\frac{\delta T}{4}\}$. Notice that, for $\im(\theta)>0$, we have
\begin{equation}\label{lulu1}\sum_{m,l=0}^{+\infty}(1-e^{i\theta})\left(-iT\partial_\theta\right)^{n-1}v_{m,l}(t,x,x_0,\sigma,\theta)=(1-e^{i\theta})\left(-iT\partial_\theta\right)^{n-1}v(t,x,x_0,\sigma,\theta).\end{equation}
Hence, $\left(-iT\partial_\theta\right)^{n-1}v$ admits a meromorphic  continuation in  $\{\theta\in\mathbb C\ :\ \im(\theta)>-\frac{\delta T}{4}\}$ with first order poles only at the points $2k\pi$, $k\in\mathbb Z$. Therefore,  upon integration of $\left(-iT\partial_\theta\right)^{n-1}v$ with respect to $\theta$ we find that in the region $t\in[0,T]$, $\sigma\in \mathbb S^{n-1}$, $x,\ x_0\in\R^n$, $\theta\in\{\theta\in\mathbb C\ :\ \im(\theta)>-\frac{\delta T}{4},\ \theta\notin2\pi\mathbb Z+i\R^-\}$ the function $v$ is analytic with respect to $\theta$ and infinitely differentiable with respect to the remaining variables. In the same way we can establish this result for $t\in[-T,T]$.  Here, for $\theta\in\{\theta\in\mathbb C\ :\ \abs{\theta}\leq\epsilon_0,\ \theta\notin2\pi\mathbb Z+i\R^-\}$ with $\epsilon_0$ sufficiently small, we have
\[v=Ce^{-\delta\left\langle x\right\rangle}e^{-\delta\left\langle x_0\right\rangle}\theta^{n-2}\log(\theta)+\tilde{v},\]
 where $\tilde{v}$ is analytic at $\theta=0$ and $C$ is constant.
 Hence  from   \eqref{lulu}, \eqref{13} and \eqref{lulu1}, we deduce that
 $F'[E(t,x,x_0)](t,\theta)$
 admits an analytic continuation, with respect to $\theta$, to
 \[\{\theta\in\mathbb C\ :\ \im(\theta)>-\frac{\delta T}{4},\ \theta\notin2\pi\mathbb Z+i\R^-\}.\] Moreover, for all $\theta\in\{\theta\in\mathbb C\ :\ \im(\theta)>-\frac{\delta T}{4},\ \theta\notin2\pi\mathbb Z+i\R^-\}$, we have
 \begin{equation}\label{lalali}F'[E(t,x,x_0)](t,\theta)=e^{-\left(\frac{\delta T+4\im(\theta)}{2T}\right)\left\langle x\right\rangle}H(t,x,x_0,\theta)e^{-\left(\frac{\delta T+4\im(\theta)}{2T}\right)\left\langle x_0\right\rangle},\end{equation}
 where $H(t,x,x_0,\theta)$ is $\mathcal C^\infty$ and bounded with respect to $t\in\R$, $x\in\R^n$ and $x_0\in\R^n$. Also,  for $\theta\in\{\theta\in\mathbb C\ :\ \abs{\theta}\leq\epsilon_0,\ \theta\notin2\pi\mathbb Z+i\R^-\}$ with $\epsilon_0$ sufficiently small, we have the representation
 \begin{equation}\label{tutu}F'[E(t,x,x_0)](t,\theta)=C\left(\int_{\mathbb S^{n-1}}e^{-\frac{i\theta\left\langle x-x_0,\sigma\right\rangle}{T}}\d\omega(\sigma)\right)e^{-\delta\left\langle x\right\rangle}e^{-\delta\left\langle x_0\right\rangle}\theta^{n-2}\log(\theta)+\tilde{E},\end{equation}
 where $\tilde{E}$ is analytic at $\theta=0$. Thus, the family of operators
 \[\sum_{m,l=0}^{+\infty}F'\left[e^{-\delta\left\langle x\right\rangle}h_{m,l}(t)\phi_m\frac{\sin(t\Lambda)}{\Lambda}\phi_le^{-\delta\left\langle x\right\rangle}\right](t,\theta)\]
 admits an analytic continuation to $\{\theta\in\mathbb C\ :\ \im(\theta)>-\frac{\delta T}{4},\ \theta\notin2\pi\mathbb Z+i\R^-\}$ which is defined by
 \[\left(\sum_{m,l=0}^{+\infty}F'\left[e^{-\delta\left\langle x\right\rangle}h_{m,l}(t)\phi_m\frac{\sin(t\Lambda)}{\Lambda}\phi_le^{-\delta\left\langle x\right\rangle}\right](t,\theta)g\right)(x)=\int_{\R^n}F'[E(t,x,x_0)](t,\theta)g(x_0)\d x_0\]
 and, for $\epsilon_0$ sufficiently small and $\theta\in\{\theta\in\mathbb C\ :\ \abs{\theta}\leq\epsilon_0,\ \theta\notin2\pi\mathbb Z+i\R^-\}$, we have the following representation
\[\sum_{m,l=0}^{+\infty}F'\left[e^{-\delta\left\langle x\right\rangle}h_{m,l}(t)\phi_m\frac{\sin(t\Lambda)}{\Lambda}\phi_le^{-\delta\left\langle x\right\rangle}\right](t,\theta)=B_1(\theta)\theta^{n-2}\log(\theta)+ C_1(t,\theta).\] 
Here $C_1(t,\theta)$  is $\mathcal C^\infty$  and $T$-periodic with respect to $t$ and analytic with respect to $\theta$  and $B_1(\theta)$ is analytic with respect to $\theta$ for $\abs{\theta}\leq\epsilon_0$. From \eqref{tutu}, the operators $\left(\partial^j_\theta B_1(\theta)\right)_{|\theta=0}$ take the following form
\[\left(\left(\partial^j_\theta B_1(\theta)\right)_{|\theta=0}\right)g=e^{-\delta\left\langle x\right\rangle}\left(P_j(x)\right)*\left(e^{-\delta\left\langle x\right\rangle}g\right),\quad j\in\mathbb N,\ g\in L^2(\R^n)\]
with $P_j$ an homogeneous polynomial of order $j$. This completes the proof.\end{proof}

Consider  the operator 
\[\cos(t\Lambda) : \dot{H}^1(\R^n)\rightarrow \dot{H}^1(\R^n)\]
with $\Lambda=\sqrt{-\Delta}$. The solution $u$ of \eqref{problem} with $V=0$, $\tau=0$, $F=0$ and $f_2=0$ is given by
\[u(t)=\cos(t\Lambda)f_1.\]

To prove the meromorphic continuation of the transformation $F'$ of all solutions of the free wave equation we must show the following.
\begin{lem}\label{l10}
Let $n\geq4$ be even. Then, the family of operators  
\[\sum_{m,l\in\mathbb N}F'\left[e^{-\delta\left\langle x\right\rangle}h_{m,l}(t)\phi_m\cos(t\Lambda)\phi_le^{-\delta\left\langle x\right\rangle}\right](t,\theta):H^1(\R^n)\rightarrow H^1(\R^n),\]
converges to a family of operators analytic  with respect to $\theta$, continuous with respect to $t\in\R$ and   satisfying the properties described in Theorem \ref{t4}
and
\[\sum_{m,l\in\mathbb N}F'\left[e^{-\delta\left\langle x\right\rangle}h_{m,l}(t)\phi_m\cos(t\Lambda)\phi_le^{-\delta\left\langle x\right\rangle}\right](t,\theta):H^1(\R^n)\rightarrow L^2(\R^n)\]
converges to a family of operators analytic  with respect to $\theta$, $\mathcal C^1$ with respect to $t\in\R$ and   satisfying the properties described in Theorem \ref{t4}.\end{lem}
\begin{proof}
 Choose $\alpha\in\mathcal C^\infty(\R)$  such that $0\leq \alpha(t)\leq 1$, $\alpha(t)=0$ for $t\leq \frac{T}{5}$ and $\alpha(t)=1$ for $t\geq \frac{T}{4}$. An application of the Duhamel's principle (see \cite{Ki3})  yields
\[\cos(t\Lambda)=\int_0^{\frac{T}{4}} \frac{\sin((t-s)\Lambda)}{\Lambda}\left[\partial_t^2,\alpha\right](s)\cos(s\Lambda)\d s,\quad t\geq T.\]
Let  $\left(\chi_m\right)_m$ be a familly of functions lying in $\CI$ such that for $m\geq2$
$\chi_m=1$ on\\
 $\{x\ :\ 2^{m-1}T-T\leq\abs{x}\leq2^{m+1}T+T\}$, $\chi_1=1$ on $\abs{x}\leq4T+T$ and $\chi_0=1$ on $\abs{x}\leq 3T$,
\[\textrm{supp}\chi_m\subset\{x\ :\ 2^{m-1}T-\frac{5T}{4}\leq\abs{x}\leq2^{m+1}T+\frac{5T}{4}\},\]
\[\textrm{supp}\chi_1\subset\{x\ :\ \abs{x}\leq4T+\frac{5T}{4}\},\]
\[\textrm{supp}\chi_0\subset\{x\ :\ \abs{x}\leq2T+\frac{5T}{4}\},\]
and 
\[\sum_{m=0}^{+\infty}\chi_m(x)\leq2.\]
An application of the finite speed of propagation yields
\begin{equation}\label{kuku}\phi_m\cos(t\Lambda)\phi_l=\int_0^{\frac{T}{4}} \phi_m\frac{\sin((t-s)\Lambda)}{\Lambda}\chi_l\left[\partial_t^2,\alpha\right](s)\cos(s\Lambda)\phi_l\d s,\quad t\geq T\end{equation}
Let $g_{m,l}\in\mathcal C^\infty(\R)$ be such that 
 \[0\leq g_{m,l}(t)\leq 1,\quad \sup_{m,l\in\mathbb N}\norm{\partial_t^jg_{m,l}}_{L^\infty(\R)}\leq C_j<\infty,\]
 \[g_{m,l}(t)=1\quad \textrm{for}\quad t\geq 2^{m+1}T+2^{l+1}T+2T+\frac{T}{2}\]
 \[g_{m,l}(t)=0\quad \textrm{for}\quad t\leq 2^{m+1}T+2^{l+1}T+2T+\frac{T}{3}.\]  
Applying the arguments of  Lemma \ref{l2}, we can show that \[\sum_{m,l=0}^{+\infty}F'\left[e^{-\delta\left\langle x\right\rangle}g_{m,l}(t)\phi_m\frac{\sin((t-s)\Lambda)}{\Lambda}\chi_le^{-\delta\left\langle x\right\rangle}\right](t,\theta)\]
admits the analytic continuation described in Lemma \ref{l2} uniformly with respect to $s\in\left[0,\frac{T}{4}\right]$.
Also, since $1-g_{m,l}(t)=0$ for $t\geq 2^{m+1}T+2^{l+1}T+3T$, reapeting the arguments used for $n$ odd we show that
\[\sum_{m,l=0}^{+\infty}F'\left[e^{-\delta\left\langle x\right\rangle}h_{m,l}(t)(1-g_{m,l}(t))\phi_m\frac{\sin((t-s)\Lambda)}{\Lambda}\chi_le^{-\delta\left\langle x\right\rangle}\right](t,\theta)\]
admits an analytic continuation to $\im(\theta)>-\frac{\delta T}{4}$   uniformly with respect to $s\in\left[0,\frac{T}{4}\right]$.
Moreover, since $\left[\partial_t^2,\alpha\right](t)=0$ for $t\notin\left[\frac{T}{5},\frac{T}{4}\right]$ we have \[e^{\delta\left\langle x\right\rangle}\left[\partial_t^2,\alpha\right](s)\cos(s\Lambda)e^{-\delta\left\langle x\right\rangle}\in\mathcal{L}(H^1(\R^n))\]
and 
\[\norm{e^{\delta\left\langle x\right\rangle}\left[\partial_t^2,\alpha\right](s)\cos(s\Lambda)\phi_le^{-\delta\left\langle x\right\rangle}}_{\mathcal{L}(H^1(\R^n))}\lesssim \norm{e^{\delta\left\langle x\right\rangle}\left[\partial_t^2,\alpha\right](s)\cos(s\Lambda)e^{-\delta\left\langle x\right\rangle}}_{\mathcal{L}(H^1(\R^n))}\lesssim1.\]

Thus, applying $F'$ to \eqref{kuku}, for $\im(\theta)>0$, we obtain
\[\begin{aligned}&F'\left[e^{-\delta\left\langle x\right\rangle}h_{m,l}(t)\phi_m\cos(t\Lambda)\phi_le^{-\delta\left\langle x\right\rangle}\right](t,\theta)\\
&=\int_0^{\frac{T}{4}} F'\left[e^{-\delta\left\langle x\right\rangle}g_{m,l}(t)\phi_m\frac{\sin((t-s)\Lambda)}{\Lambda}\chi_le^{-\delta\left\langle x\right\rangle}\right](t,\theta)e^{\delta\left\langle x\right\rangle}\left[\partial_t^2,\alpha\right](s)\cos(s\Lambda)\phi_le^{-\delta\left\langle x\right\rangle}\d s\\
&\ \ \ +\int_0^{\frac{T}{4}} F'\left[e^{-\delta\left\langle x\right\rangle}h_{m,l}(t)(1-g_{m,l}(t))\phi_m\frac{\sin((t-s)\Lambda)}{\Lambda}\chi_le^{-\delta\left\langle x\right\rangle}\right](t,\theta)e^{\delta\left\langle x\right\rangle}\left[\partial_t^2,\alpha\right](s)\cos(s\Lambda)\phi_le^{-\delta\left\langle x\right\rangle}\d s.\end{aligned}\]
From this last representation we deduce the analytic continuation of
\[\sum_{m,l=0}^{+\infty}F'\left[e^{-\delta\left\langle x\right\rangle}h_{m,l}(t)\phi_m\cos(t\Lambda)\phi_le^{-\delta\left\langle x\right\rangle}\right](t,\theta).\]\end{proof}

\textit{Proof of Theorem \ref{t4}.}
We obtain the analytic continuation of $F'\left[e^{-\delta\left\langle x\right\rangle}U_0(t) e^{-\delta\left\langle x\right\rangle}\right](t,\theta)$ by combining Lemma \ref{l9}, Lemma \ref{l2} and Lemma \ref{l10} (see also the proof of Lemma 2 in \cite{Ki4}).\qed

Applying Theorem \ref{t4} to the representation \eqref{6}, we show easily the following (see also Lemma 6 in \cite{Vai}).

\begin{cor}\label{c1}
The family of operators
 \[e^{-\delta\left\langle x\right\rangle}J(t,\theta)e^{-\delta\left\langle x\right\rangle}:\mathcal{H}_{1,per}(\R^{1+n})\rightarrow \mathcal{H}_{1,per}(\R^{1+n})\]
  admits an analytic continuation with respect to $\theta$,  from 
 $\{\theta\in\C\ :\ \im(\theta)>0\}$ to $\{ \theta\in\mathbb C\ :\ \im(\theta)>-\frac{\delta T}{4}\}$ for $n$ odd and to \\
 $\{ \theta\in\mathbb C\ :\ \im(\theta)>-\frac{\delta T}{4},\ \theta\notin2\pi\mathbb Z+i\R^-\}$ for $n$ even. Moreover, for $n$ even, there exists $\epsilon_0>0$ such that for $\theta\in\{ \theta\in\mathbb C\ :\ \theta\notin2\pi\mathbb Z+i\R^-,\ \abs{\theta}\leq \epsilon_0\}$  we have the representation
 \[e^{-\delta\left\langle x\right\rangle}J(t,\theta)e^{-\delta\left\langle x\right\rangle}=B_2(\theta)\theta^{n-2}\log(\theta)+C_2(t,\theta),\]
 where $\log$ is the logarithm defined on $\C\setminus i\R^-$,  $C_2(t,\theta)$ is $\mathcal C^\infty$  and $T$-periodic with respect to $t$ and analytic with respect to $\theta$, $B_2(\theta)$ is analytic with respect to $\theta$ for $\abs{\theta}\leq\epsilon_0$ for $\abs{\theta}\leq \epsilon_0$,  for all  $l\in\mathbb N$  $\left(\partial_\theta^l\left(B_2\right)(\theta)\right)_{|\theta=0}$ are finite rank operators.
\end{cor}

\section{The meromorphic continuation of the resolvent}
In this section we will apply the results of Sections 1 and 2 to  show the meromorphic continuation of $R(\theta)$ described in Theorem \ref{t4}. 
According to Corollary \ref{c1} and Proposition \ref{p1}, if we show that there exists $\theta_0$ such that  for all $G\in \mathcal{H}_{1,per}(\R^{1+n})$ the equation
\begin{equation}\label{14}g+e^{-\delta\left\langle x\right\rangle}J(t,\theta_0)Q(t)e^{\delta\left\langle x\right\rangle}g=G\end{equation}
has a solution $g\in\mathcal{H}_{1,per}(\R^{1+n})$, since $e^{-\delta\left\langle x\right\rangle}J(t,\theta_0)Q(t)e^{\delta\left\langle x\right\rangle}$ is a compact operator, the operator
\[\textrm{Id}+e^{-\delta\left\langle x\right\rangle}J(t,\theta_0)Q(t)e^{\delta\left\langle x\right\rangle}\] will be invertible and we can conclude by applying Theorem 8 of \cite{Vai} and the analytic Fredholm Theorem. In order to solve \eqref{14}, we introduce the following Hilbert spaces.
\begin{defi}
Let $I$ be an interval of $\R$. Then, we denote by $\mathcal H_{\gamma}(I\times\R^n)$ the space
\[\mathcal H_{\gamma}(I\times\R^n)=H^{\gamma}(I\times\R^n)\times H^{\gamma-1}(I\times\R^n).\]
\end{defi}
We can easily prove that there exists $r>0$ such that  $W_0\in\mathcal L\left(\mathcal H_{2}^r(\R^{1+n})\right)$. To solve \eqref{14} we will use the following result that we show  by  applying  some arguments of Vainberg (see Theorem 5 in \cite{Vai}).
\begin{Thm}\label{t8}\label{l3}
There exists $A>r$ such that for all $h\in\mathcal H_{1}^r(\R^{1+n})$ the equation
\begin{equation}\label{15}\phi(t)+e^{-\delta\left\langle x\right\rangle}W_0\left(Q(t)e^{\delta\left\langle x\right\rangle}\phi(t)\right)(t)=h(t)\end{equation}
admits a unique solution $\phi\in\mathcal H_{1}^A(\R^{1+n})$.\end{Thm}
\begin{proof}
For all $A>r$, let  $\phi=(\phi_1,\phi_2)\in \mathcal H_{1}^A(\R^{1+n})$ be such that
\[\phi(t)+e^{-\delta\left\langle x\right\rangle}\int_0^tU_0(t-s)Q(s)e^{\delta\left\langle x\right\rangle}\phi(s)\d s=0.\]
Since $\int_0^tU_0(t-s)Q(s)e^{\delta\left\langle x\right\rangle}\phi(s)\d s\in\mathcal H_{1}^A(\R^{1+n})$, we get $e^{\delta\left\langle x\right\rangle}\phi\in\mathcal H_{1}^A(\R^{1+n})$ and
\[e^{\delta\left\langle x\right\rangle}\phi(t)=-\int_0^tU_0(t-s)Q(s)e^{\delta\left\langle x\right\rangle}\phi(s)\d s\]
Thus, $Z=e^{\delta\left\langle x\right\rangle}\phi_1$ will be the solution of
\[  \left\{\begin{aligned}
\partial_t^2(Z)-\Delta_{x}Z+V(t,x)Z&=0,\ \ (t,x)\in{\R}\times{\R}^{n},\\
(Z, \partial_t(Z))(0,x)&=(0,0),\ \ x\in{\R}^n,\end{aligned}\right.\]
and $e^{\delta\left\langle x\right\rangle}\phi_2=\partial_t(Z)$. Hence, $\phi=(0,0)$ and \eqref{15} admits a unique solution in $\mathcal H_{1}^A(\R^{1+n})$. For the proof of Theorem \ref{t8} it now suffices to show that \eqref{15} has a solution for $h$ on a dense set of functions in $\mathcal H_{1}^r(\R^{1+n})$ and that this solution satisfies
\begin{equation}\label{40}\norm{\phi}_{\mathcal H_{1}^A(\R^{1+n})}\lesssim \norm{h}_{\mathcal H_{1}^r(\R^{1+n})},\end{equation}
for some $A>r$. We take $\mathcal C^\infty(\R^{1+n})\times\mathcal C^\infty(\R^{1+n})\cap e^{-\delta\left\langle x\right\rangle}\mathcal H_{1}^r(\R^{1+n})$ as our dense set. Clearly,   we  have
\[\left(\textrm{Id}+W_0Q(t)\right)\phi\in\mathcal C^\infty(\R^{1+n})\times\mathcal C^\infty(\R^{1+n})\cap \mathcal H_{1}^r(\R^{1+n}),\quad\phi\in\mathcal C^\infty(\R^{1+n})\times\mathcal C^\infty(\R^{1+n})\cap \mathcal H_{1}^r(\R^{1+n}).\]
Let $h\in\mathcal C^\infty(\R^{1+n})\times\mathcal C^\infty(\R^{1+n})\cap e^{-\delta\left\langle x\right\rangle}\mathcal H_{1}^r(\R^{1+n})$ and consider the equation 
\begin{equation}\label{50}\phi_1+W_0Q(t)\phi_1=e^{\delta\left\langle x\right\rangle}h.\end{equation}
Equation \eqref{50} is a Volterra equation and it is uniquely solvable in $\mathcal C^\infty(\R^{1+n})\times\mathcal C^\infty(\R^{1+n})\cap \mathcal H_{1}^r(\R^{1+n})$. Now choose $\phi=e^{-\delta\left\langle x\right\rangle}\phi_1$. We obtain
 \[e^{\delta\left\langle x\right\rangle}\phi+W_0Q(t)e^{\delta\left\langle x\right\rangle}\phi=e^{\delta\left\langle x\right\rangle}h\]
 and it follows that $\phi$ is the solution of \eqref{15}. It remains to show \eqref{40} and conclude by density.
Let 
\[\psi=e^{-rt}\phi,\quad \hat{U}_0(t)=e^{-rt}U_0(t).\]
Then, \eqref{15} can be rewritten in the form
\begin{equation}\label{41}\psi+e^{-\delta\left\langle x\right\rangle}\hat{W}_0Q(t)e^{\delta\left\langle x\right\rangle}\psi=e^{-rt}h,\quad\textrm{where }(\hat{W}_0Q(t)\psi)(t)=\int_0^t\hat{U}_0(t-s)Q(s)\psi(s)\d s.\end{equation}
Repeating the arguments used in Proposition \ref{p1}, we can easily prove that for any $g\in \mathcal H_{1}(\R^{1+n})$, we get
\[\norm{e^{-\delta\left\langle x\right\rangle}\hat{W}_0Q(t)e^{\delta\left\langle x\right\rangle}g}_{\mathcal H_{2}(\R^{1+n})}\lesssim\norm{g}_{\mathcal H_{1}(\R^{1+n})}.\]
Now let $T_1>0$. We have
\begin{equation}\label{42}\norm{e^{-\delta\left\langle x\right\rangle}\hat{W}_0Q(t)e^{\delta\left\langle x\right\rangle}g}_{\mathcal H_{2}([0,T_1]\times\R^{n})}\lesssim\norm{g}_{\mathcal H_{1}(\R^{1+n})}.\end{equation}
Suppose that the function $\tilde{\psi}$ is equal to $\psi$ for $t<T_1$ and
\begin{equation}\label{43}\norm{\tilde{\psi}}_{\mathcal H_{1}(\R^{1+n})}\leq 2\norm{\psi}_{\mathcal H_{1}([0,T_1]\times\R^{n})}.\end{equation}
Clearly, such a function exists (it is constructed by the means of extension operator). We recall that $\psi\in\mathcal C^\infty(\R^{1+n})\times\mathcal C^\infty(\R^{1+n})\cap \mathcal H_1(\R^{1+n})$ and, consequently, the right-hand side of \eqref{43} is bounded. The left hand-side of \eqref{42} is independent of the values of $g(t)$ for $t>T_1$. It therefore follows from \eqref{42} for $g=\tilde{\psi}$ and \eqref{43} that
\begin{equation}\label{44}\norm{e^{-\delta\left\langle x\right\rangle}\hat{W}_0Q(t)e^{\delta\left\langle x\right\rangle}\psi}_{\mathcal H_{2}([0,T_1]\times\R^{n})}\leq 2C_1\norm{\psi}_{\mathcal H_{1}([0,T_1]\times\R^{n})}.\end{equation}
For all $h=(h_1,h_2)\in\C^2$, we set $(h)_1=h_1$. We denote by $E(T_1,g)$ the quantity
\[E(T_1,g)=\norm{g_1(T_1)}_{H^1(\R^n)}^2+\norm{\partial_t\left(g_1\right)(T_1)}_{L^2(\R^n)}^2+\norm{g_2(T_1)}_{L^2(\R^n)}^2,\]
where $g=(g_1,g_2)$. Notice that $e^{-\delta\left\langle x\right\rangle}\hat{W}_0Q(t)e^{\delta\left\langle x\right\rangle}\psi\in\mathcal C^\infty(\R^{1+n})\times\mathcal C^\infty(\R^{1+n})\cap\mathcal H_{2}([0,T_1]\times\R^{n})$,\\
 $\left[\left(e^{-\delta\left\langle x\right\rangle}\hat{W}_0Q(t)e^{\delta\left\langle x\right\rangle}\psi\right)_1\right]_{t=0}=0$ and $\left[\partial_t\left(e^{-\delta\left\langle x\right\rangle}\hat{W}_0Q(t)e^{\delta\left\langle x\right\rangle}\psi\right)_1\right]_{t=0}=0$. Thus, we have
 \[E(0,e^{-\delta\left\langle x\right\rangle}\hat{W}_0Q(t)e^{\delta\left\langle x\right\rangle}\psi)=0.\]
Then, for $\psi$ compactly supported in $x$, we obtain
 \[E(T_1,e^{-\delta\left\langle x\right\rangle}\hat{W}_0Q(t)e^{\delta\left\langle x\right\rangle}\psi)\leq\int_0^{T_1}\abs{\partial_t\left[E(t,e^{-\delta\left\langle x\right\rangle}\hat{W}_0Q(t)e^{\delta\left\langle x\right\rangle}\psi)\right]}\d t\]
 and an application of the Hölder's inequality yields
 \[E(T_1,e^{-\delta\left\langle x\right\rangle}\hat{W}_0Q(t)e^{\delta\left\langle x\right\rangle}\psi)\lesssim\norm{e^{-\delta\left\langle x\right\rangle}\hat{W}_0Q(t)e^{\delta\left\langle x\right\rangle}\psi}^2_{\mathcal H_{2}([0,T_1]\times\R^{n})}.\]
 By density, we can extend this estimate to all $\psi\in\mathcal C^\infty(\R^{1+n})\times\mathcal C^\infty(\R^{1+n})\cap \mathcal H_1(\R^{1+n})$.
Since $\psi(0)=0$, combining this result with \eqref{44}, we obtain
\begin{equation}\label{45}E(T_1,e^{-\delta\left\langle x\right\rangle}\hat{W}_0Q(t)e^{\delta\left\langle x\right\rangle}\psi)\leq C\norm{\psi}^2_{\mathcal H_{1}([0,T_1]\times\R^{n})}=C\int_0^{T_1}E(t,\psi)\d t\end{equation}
with $C>0$ independent of $\psi$ and $T_1$.
Clearly, for any $\phi_1$, $\phi_2$ we have
\[E((T_1,\phi_1+\phi_2)\leq 2\left[E((T_1,\phi_1)+E((T_1,\phi_2)\right].\]
Consequently, it follows from \eqref{41} and \eqref{45} that
\[E(T_1,\psi)=E(T_1,e^{-rt}h-e^{-\delta\left\langle x\right\rangle}\hat{W}_0Q(t)e^{\delta\left\langle x\right\rangle}\psi)\leq 2C\int_0^{T_1}E(t,\psi)\d t+2E(T_1,e^{-rt}h).\]
Hence, by the Gronwall's lemma we get
\[E(T_1,\psi)=2E(T_1,e^{-rt}h)e^{2CT_1}.\]
If we now multiply this inequality by $e^{-CT_1}$and integrate with respect to $T_1>0$ we obtain \eqref{40} with $A=\sqrt{2C}+r$.\end{proof}
\ \\
\textit{Proof of Theorem \ref{t3}.} As it was mentioned in the beginning of this section, it remains to solve \eqref{14} for some $\theta_0\in\C$ with $\im(\theta_0)>rT$. Choose $\theta_0=i(A+1)T$ with $A$ the constant of Theorem \ref{t8}. 
Let $G\in \mathcal{H}_{1,per}(\R^{1+n})$ and let $\alpha\in\mathcal C^\infty(\R)$ be such that $0\leq \alpha(t)\leq1$, $\alpha(t)=0$ for $t\leq \frac{T}{2}$ and $\alpha(t)=1$ for $t\geq \frac{2T}{3}$. Then, $h(t)=\alpha(t)G(t)\in \mathcal H_{1}^r(\R^{1+n})$. Let $\phi$ be the solution of \eqref{15} with $h(t)=\alpha(t)G(t)$. We have
\begin{equation}\label{46}\phi(t)+e^{-\delta\left\langle x\right\rangle}W_0\left(Q(t)e^{\delta\left\langle x\right\rangle}\phi(t)\right)(t)=\alpha(t)G(t).\end{equation}
According to Proposition \ref{ro}, since $\phi\in\mathcal H_{1}^A(\R^{1+n})$, applying $F'$ to both sides of \eqref{46}  at $\theta=\theta_0$, we obtain
\begin{equation}\label{16}F'\left[\phi\right](t,\theta_0)+e^{-\delta\left\langle x\right\rangle}J(t,\theta_0)Q(t)e^{\delta\left\langle x\right\rangle}F'\left[\phi\right](t,\theta_0)=F'\left[\alpha(t)G(t)\right](t,\theta_0).\end{equation}
Since $G(t)$ is $T$-periodic with respect to $t$, for $0\leq t< T$, we have
\[\begin{aligned}F'\left[\alpha(t)G(t)\right](t,\theta_0)&=\left(e^{-(A+1)t}\sum_{k=0}^{+\infty}\alpha(t+kT)e^{-(A+1)kT}\right)G(t)\\
&=e^{-(A+1)t}\left(\alpha(t)+\frac{e^{-(A+1)T}}{\left(1-e^{-(A+1)T}\right)}\right)G(t).\end{aligned}\]
Let $p_1(t)$ be the $\mathcal C^\infty$ and $T$-periodic function defined by
\[p_1(t)=e^{-(A+1)t}\left(\alpha(t)+\frac{e^{-(A+1)T}}{\left(1-e^{-(A+1)T}\right)}\right),0\leq t<T.\]
Since $\alpha(t)\geq0$, we have $p_1(t)>0$ and according to \eqref{16},
\[g(t)=\frac{F'\left[\phi(t)\right](t,\theta_0)}{p_1(t)}\]
is a  solution of \eqref{14}. Thus, from Proposition \ref{p1}, we deduce that 
\[\textrm{Id}+e^{-\delta\left\langle x\right\rangle}J(t,\theta_0)Q(t)e^{\delta\left\langle x\right\rangle}\] 
is invertible. Then, applying the analytic Fredholm theorem, we obtain that \[F'(e^{-\delta\left\langle x\right\rangle}U(t,0)e^{-\delta\left\langle x\right\rangle})(t,\theta):\ \Bb\to\Bb\] admits the following meromorphic continuation
\[F'\left[e^{-\delta\left\langle x\right\rangle}U(t,0)e^{-\delta\left\langle x\right\rangle}\right](t,\theta)=\left(\textrm{Id}+e^{-\delta\left\langle x\right\rangle}J(t,\theta)Q(t)e^{\delta\left\langle x\right\rangle}\right)^{-1}F'(e^{-\delta\left\langle x\right\rangle}U_0(t)e^{-\delta\left\langle x\right\rangle})(t,\theta).\]
Moreover, applying Theorem 8 of \cite{Vai}, for $n$ even and $\theta\in\{ \theta\in\mathbb C\ :\ \theta\notin2\pi\mathbb Z+i\R^-,\ \abs{\theta}\leq \epsilon_0\}$, we obtain the following representation 
\[F'\left[e^{-\delta\left\langle x\right\rangle}U(t,0)e^{-\delta\left\langle x\right\rangle}\right](t,\theta)=\theta^{-m}\sum_{j\geq0}\left(\frac{\theta}{R_t(\log \theta)}\right)^jP_{j,t}(\log \theta)+C(t,\theta),\]
where $C(t,\theta)$ is analytic with respect to $\theta$, $R_{t}$ is a polynomial, the $P_{j,t}$ are polynomials of order at most $l_j$ and $\log$ is the logarithm defined on $\mathbb{C}\setminus i{\R}^-$. Also, $C(t,\theta)$ and the coefficients of  the polynomials $R_{t}$ and $P_{j,t}$ are $\mathcal C^\infty$ and $T$-periodic with respect to $t$. We conclude by applying \eqref{2} and Lemma \ref{l8}.\qed

\section{Applications to local energy decay and Strichartz estimates} 
In  this section, we  apply the result of Theorem \ref{t3} to prove estimates \eqref{local} and \eqref{1}. We start by showing that assumption $\rm(H3)$ implies the local energy decay \eqref{local}. Then, by combining this result with some arguments of \cite{P2} and \cite{SS}, we establish the global Strichartz estimates \eqref{1}. In Subsection 4.3 we give examples of potential $V(t,x)$ such that $\rm(H3)$ is fulfilled.
\subsection{Local energy decay}

The goal of this subsection is to prove the local energy decay introduced in Theorem \ref{t1}. For this purpose, we need the following.
\begin{lem}\label{l4}
Assume $\rm(H1)$ and $\rm(H2)$ fulfilled. Then, for all $0\leq t\leq T$, 
\[e^{- (D+\epsilon)\left\langle x\right\rangle}U(t,0)e^{ D\left\langle x\right\rangle},\ e^{ D\left\langle x\right\rangle}U(t,0)e^{-(D+\epsilon)\left\langle x\right\rangle}\in\mathcal L(\B)\]
and we have
\begin{equation}\label{102}\norm{e^{- (D+\epsilon)\left\langle x\right\rangle}U(t,0)e^{ D\left\langle x\right\rangle}}_{\mathcal L(\B)}\leq C,\end{equation}
\begin{equation}\label{103}\norm{e^{ D\left\langle x\right\rangle}U(t,0)e^{-(D+\epsilon)\left\langle x\right\rangle}}_{\mathcal L(\B)}\leq C\end{equation}
with $C$ independent of $t$.\end{lem}
\begin{proof}
Let $\phi\in\CI\times\CI$.
Applying the Duhamel's principle we obtain these representations
\begin{equation}\label{18}U(t,0)=U_0(t)-\int_0^tU(t,s)Q(s)U_0(s)\d s,\end{equation}
\begin{equation}\label{19}U(t,0)=U_0(t)-\int_0^tU_0(t-s)Q(s)U(s,0)\d s.\end{equation}
It is well known that $e^{\mp (D+\frac{\epsilon}{2})\left\langle x\right\rangle}U_0(t)e^{\pm \left(D+\frac{\epsilon}{2}\right)\left\langle x\right\rangle}\in\mathcal L(\Bb)$ and
\[C_1=\sup_{t\in [0,T]}\norm{e^{\mp (D+\frac{\epsilon}{2})\left\langle x\right\rangle}U_0(t)e^{\pm (D+\frac{\epsilon}{2})\left\langle x\right\rangle}}_{\mathcal L(\Bb)}<+\infty.\]
Since the multiplication by $e^{- \frac{\epsilon}{2}\left\langle x\right\rangle}$ is continuous from $\B$ to $\Bb$ we have 
\[e^{- (D+\epsilon)\left\langle x\right\rangle}U_0(t)e^{D\left\langle x\right\rangle},\ e^{D\left\langle x\right\rangle}U_0(t)e^{-(D+\epsilon)\left\langle x\right\rangle}\in\mathcal L(\B)\]
and, for all $0\leq t\leq T$, we get
\begin{equation}\label{100}\norm{e^{- (D+\epsilon)\left\langle x\right\rangle}U_0(t)e^{D\left\langle x\right\rangle}}_{\mathcal L(\B)}\leq C_1'\norm{e^{- (D+\frac{\epsilon}{2})\left\langle x\right\rangle}U_0(t)e^{ (D+\frac{\epsilon}{2})\left\langle x\right\rangle}}_{\mathcal L(\Bb)}\leq C_2,\end{equation}
\begin{equation}\label{101}\norm{e^{D\left\langle x\right\rangle}U_0(t)e^{-(D+\epsilon)\left\langle x\right\rangle}}_{\mathcal L(\B)}\leq C_1'\norm{e^{ (D+\frac{\epsilon}{2})\left\langle x\right\rangle}U_0(t)e^{- (D+\frac{\epsilon}{2})\left\langle x\right\rangle}}_{\mathcal L(\Bb)}\leq C_2.\end{equation}
Hence, multiplying the representation \eqref{18} on the right  by $e^{- (D+\epsilon)\left\langle x\right\rangle}$ and on the left by $e^{ D\left\langle x\right\rangle}$ , we obtain
\[\begin{aligned}e^{- (D+\epsilon)\left\langle x\right\rangle}U(t,0)e^{ D\left\langle x\right\rangle}\phi=&e^{ -(D+\epsilon)\left\langle x\right\rangle}U_0(t)e^{ D\left\langle x\right\rangle}\phi\\
\ &-\int_0^t\left(e^{- (D+\epsilon)\left\langle x\right\rangle}U(t,s)Q(s)e^{ (D+\epsilon)\left\langle x\right\rangle}\right)e^{- (D+\epsilon)\left\langle x\right\rangle}U_0(s)e^{ D\left\langle x\right\rangle}\phi\d s.\end{aligned}\]
Then applying (H1), we have $e^{- (D+\epsilon)\left\langle x\right\rangle}U(t,s)Q(s)e^{ (D+\epsilon)\left\langle x\right\rangle}\in\mathcal L(\B)$ and, from \eqref{100} we deduce , for all $ 0\leq t\leq T$, the estimate
\[\norm{e^{- (D+\epsilon)\left\langle x\right\rangle}U(t,0)e^{ D\left\langle x\right\rangle}\phi}_{\B}\leq \left(C_2+C_2\int_0^t\norm{e^{- (D+\epsilon)\left\langle x\right\rangle}U(t,s)Q(s)e^{ (D+\epsilon)\left\langle x\right\rangle}}_{\mathcal L(\B)}\d s\right)\norm{\phi}.\]
It follows
\[\norm{e^{- (D+\epsilon)\left\langle x\right\rangle}U(t,0)e^{ D\left\langle x\right\rangle}\phi}_{\B}\leq \left(C_2+C_3T\right)\norm{\phi}_{\B},\quad 0\leq t\leq T.\]
Thus, by density $e^{- (D+\epsilon)\left\langle x\right\rangle}U(t,0)e^{ D\left\langle x\right\rangle}\in\mathcal L(\B)$ and
\[\norm{e^{- (D+\epsilon)\left\langle x\right\rangle}U(t,0)e^{ D\left\langle x\right\rangle}}_{\mathcal L(\B)}\leq C.\]
Combining \eqref{19} and \eqref{101} with the same arguments, we obtain \[e^{- (D+\epsilon)\left\langle x\right\rangle}U(t,0)e^{ D\left\langle x\right\rangle}\in \mathcal L(\B)\] and
\[\norm{e^{- (D+\epsilon)\left\langle x\right\rangle}U(t,0)e^{ D\left\langle x\right\rangle}}_{\mathcal L(\B)}\leq C_2+C_4T,\quad 0\leq t\leq T.\]
\end{proof}
\ \\
\textit{Proof of Theorem \ref{t1}.} From the results of Theorem \ref{t3} and assumption (H3), by applying the arguments of the proof of Lemma 3 and Lemma 4 of \cite{Ki4}, we obtain
\[\norm{e^{- D\left\langle x\right\rangle}U(kT,0)e^{- D\left\langle x\right\rangle}}_{\mathcal L(\B)}\lesssim p(kT),\quad k\in\mathbb N.\]
Combining this result with Lemma \ref{l4} we obtain \eqref{local} (see the proof of Lemma 3 and Lemma 4 in \cite{Ki4}).\qed

\subsection{Global Strichartz estimates}
The goal of this subsection is to prove Theorem \ref{t2}. For this purpose we need to show the $L^2$-integrability of the local energy (see \cite{B}, \cite{P2} and \cite{Met}) which takes the following form.
\begin{prop}\label{p2} Assume $\rm(H1)$, $\rm(H2)$ and $\rm(H3)$ fulfilled and let $n\geq3$, $0\leq \gamma\leq 1$. Let $f\in\dot{\mathcal{H}}_\gamma({\R}^n)$ and
 $F\in e^{-(\frac{3D}{2}+\epsilon)\left\langle x\right\rangle}L^2_t\left(\R^+,\dot{H}_x^\gamma({\R}^n)\right)$. Then the solution $u$  of \eqref{problem} with $\tau=0$ satisfies the estimate
 \begin{equation}\label{20}\int_0^{+\infty}\norm{\left(e^{-(D+2\epsilon)\left\langle x\right\rangle}u(t), e^{-(D+2\epsilon)\left\langle x\right\rangle}u_t(t)\right)}^2_{\dot{\mathcal{H}}_\gamma({\R}^n)}\d t\lesssim \left(\norm{f}_{\dot{\mathcal{H}}_\gamma({\R}^n)}+\norm{e^{(\frac{3D}{2}+\epsilon)\left\langle x\right\rangle}F}_{L^2_t\left(\R^+,\dot{H}_x^\gamma({\R}^n)\right)}\right)^2\end{equation}
with $C$ only depending on $n$ and $\gamma$.\end{prop}
In fact estimate \eqref{20} is the main point in the proof of Theorem \ref{t2}. Combining \eqref{20} with some arguments of \cite{P2}, including applications of the Kriest-Chisliev and the Duhamel's principle, we establish estimates \eqref{1}. We start with the proof of Proposition \ref{p2}.

\ \\
\textit{Proof of Proposition \ref{p2}.}
First notice that for the free wave equation, $f\in\dot{\mathcal{H}}_\gamma({\R}^n)$ and $r>0$ we have
\begin{equation}\label{21}\int_\R\norm{e^{-r\left\langle x\right\rangle}U_0(t)f}^2_{\mathcal{H}_\gamma({\R}^n)}\d t\lesssim\norm{f}^2_{\dot{\mathcal{H}}_\gamma({\R}^n)}.\end{equation}
To obtain this estimate for $0\leq \gamma\leq 1$ we can apply a result of Smith and Sogge.
\begin{lem}\label{l5}\emph{(\cite{SS},  Lemma 2.2)}
 Let $\gamma\leq \frac{n-1}{2}$ and let  $\phi\in \mathcal S(\R^n)$. Then
\begin{equation}\label{22}\int_{\R}\Vert \phi e^{\pm it\Lambda}f\Vert_{H^\gamma({\R}^n)}^2\textrm{d}t\leq C(\phi,n,\gamma)\Vert f\Vert_{\dot{H}^\gamma({\R}^n)}^2.\end{equation}
\end{lem}
In \cite{SS} the authors consider only odd dimensions $n\geq3$ and $\phi\in\CI$, but the proof of this lemma goes without any change for even dimensions.
Moreover, to prove \eqref{22} for $\phi\in\CI$, \cite{SS} consider the following 
 \[\int_{\R}\Vert \phi e^{\pm it\Lambda}f\Vert^2_{H^\gamma({\R}^n)}\d t=\int_0^{+\infty}\int_{\R^n}\left(1+\abs{\xi}^2\right)^{\gamma}\abs{\int_{\R^n} \hat{\phi}(\xi-\eta)\hat{f}(\eta)\delta(\tau-\abs{\eta})\d\eta}^2\d\xi\ \d\tau\]
and they only exploit the fact that $\hat{\phi}\in \mathcal S(\R^n)$. Thus, \eqref{22} is still true for $\phi\in \mathcal S(\R^n)$.  Setting $(u_0(t),\partial_t(u_0)(t))=U_0(t)f$ we have the
representation
\[u_0(t)=\cos(t\Lambda)f_1+\frac{\sin(t\Lambda)}{\Lambda}f_2\]
and, since $e^{-r\left\langle x\right\rangle}\in \mathcal S(\R^n)$, \eqref{21} follows immediately. 

Passing to the estimate of $e^{-(D+2\epsilon)\left\langle x\right\rangle}U(t,0)f$, we obtain from \eqref{18} that
\[\begin{aligned}\norm{e^{-(D+2\epsilon)\left\langle x\right\rangle}U(t,0)f}_{\dot{\mathcal{H}}_\gamma({\R}^n)}\leq& \norm{e^{-(D+2\epsilon)\left\langle x\right\rangle}U_0(t)f}_{\dot{\mathcal{H}}_\gamma({\R}^n)}\\
&+\norm{\int_0^te^{-(D+2\epsilon)\left\langle x\right\rangle}U(t,s)e^{-(D+\epsilon)\left\langle x\right\rangle}\left(e^{(D+\epsilon)\left\langle x\right\rangle}Q(s)U_0(s)f \right)\d s}_{\dot{\mathcal{H}}_\gamma({\R}^n)}.\end{aligned}\]
Notice that for $0\leq \gamma\leq1$
\[\begin{array}{l}\norm{\int_0^te^{-(D+2\epsilon)\left\langle x\right\rangle}U(t,s)e^{-(D+\epsilon)\left\langle x\right\rangle}\left(e^{(D+\epsilon)\left\langle x\right\rangle}Q(s)U_0(s)f \right)\d s}_{\dot{\mathcal{H}}_\gamma({\R}^n)}\\
\leq\norm{e^{-\epsilon\left\langle x\right\rangle}\int_0^te^{-(D+\epsilon)\left\langle x\right\rangle}U(t,s)e^{-(D+\epsilon)\left\langle x\right\rangle}\left(e^{(D+\epsilon)\left\langle x\right\rangle}Q(s)U_0(s)f \right)\d s}_{\Bb}.\end{array}\]
Since the multiplication by $e^{-\epsilon\left\langle x\right\rangle}$ is continuous from $\B$ to $\Bb$, it follows
\[\begin{array}{l}\norm{\int_0^te^{-(D+2\epsilon)\left\langle x\right\rangle}U(t,s)e^{-(D+\epsilon)\left\langle x\right\rangle}\left(e^{(D+\epsilon)\left\langle x\right\rangle}Q(s)U_0(s)f \right)\d s}_{\dot{\mathcal{H}}_\gamma({\R}^n)}\\
\lesssim \norm{\int_0^te^{-(D+\epsilon)\left\langle x\right\rangle}U(t,s)e^{-(D+\epsilon)\left\langle x\right\rangle}\left(e^{(D+\epsilon)\left\langle x\right\rangle}Q(s)U_0(s)f \right)\d s}_{\B}.\end{array}\]
Thus, we obtain
\[\begin{array}{l}\norm{\int_0^te^{-(D+2\epsilon)\left\langle x\right\rangle}U(t,s)e^{-(D+\epsilon)\left\langle x\right\rangle}\left(e^{(D+\epsilon)\left\langle x\right\rangle}Q(s)U_0(s)f \right)\d s}_{\dot{\mathcal{H}}_\gamma({\R}^n)}\\
\ \\
\lesssim \int_0^t\norm{e^{-(D+\epsilon)\left\langle x\right\rangle}U(t,s)e^{-(D+\epsilon)\left\langle x\right\rangle}}_{\mathcal L(\B)}\norm{\left(e^{(D+\epsilon)\left\langle x\right\rangle}Q(s)U_0(s)f \right)}_{\B}\d s\end{array}\]
and applying Theorem \ref{t1} it follows
\[\begin{array}{l}\norm{\int_0^te^{-(D+2\epsilon)\left\langle x\right\rangle}U(t,s)e^{-(D+\epsilon)\left\langle x\right\rangle}\left(e^{(D+\epsilon)\left\langle x\right\rangle}Q(s)U_0(s)f \right)\d s}_{\dot{\mathcal{H}}_\gamma({\R}^n)}\\
\ \\
\lesssim\left(p(t)\mathds{1}_{\R^+}(t)\right)*\left(\mathds{1}_{\R^+}(t)\norm{\left(e^{(D+\epsilon)\left\langle x\right\rangle}Q(t)U_0(t)f \right)}_{\mathcal{H}_{\gamma+1}({\R}^n)}\right)(t),\quad t>0.\end{array}\]
It is clear that \eqref{21} and (H1) imply
\begin{equation}\label{23}\int_0^{+\infty}\norm{e^{(D+\epsilon)\left\langle x\right\rangle}Q(t)U_0(t)f }^2_{\mathcal{H}_{\gamma+1}({\R}^n)} \d t=\int_0^{+\infty}\norm{e^{(D+\epsilon)\left\langle x\right\rangle}V(t,x)u_0(t)}^2_{H^\gamma(\R^n)}\d t\lesssim\norm{f}^2_{\dot{\mathcal{H}}_\gamma({\R}^n)}.\end{equation}
Since $p(t)\mathds{1}_{\R^+}(t)\in L^1\left(\R\right)$, an application of the Young inequality for the convolution, combined
with  \eqref{21} and \eqref{23}, yields \eqref{20} with $F=0$. 

In the general case ($F\neq0$) consider the solution $v$ of \eqref{problem} with $\tau=0$, $f=(0,0)$ and\\ $F\in e^{-(\frac{3D}{2}+2\epsilon)\left\langle x\right\rangle}L^2_t\left(\R,\dot{H}_x^\gamma({\R}^n)\right)$. Then 
\[\left( e^{-(D+2\epsilon)\left\langle x\right\rangle}v(t), e^{-(D+2\epsilon)\left\langle x\right\rangle}v_t(t)\right)=\int_0^t e^{-(D+2\epsilon)\left\langle x\right\rangle}U(t,s) e^{-(D+\epsilon)\left\langle x\right\rangle}\left(0, e^{(D+\epsilon)\left\langle x\right\rangle}F(s)\right)\d s.\]
Clearly, we have 
\[\left(0, e^{(D+\epsilon)\left\langle x\right\rangle}F(t)\right)\in L^2_t\left(\R^+,e^{-\frac{D}{4}\left\langle x\right\rangle}\dot{\mathcal{H}}_{\gamma+1}({\R}^n)\right)\hookrightarrow L^2_t\left(\R^+,\Bb\right)\hookrightarrow L^2_t\left(\R^+,\B\right).\]
Exploiting the local energy decay of Theorem \ref{t1} and repeating the above arguments we get for $u=v$ estimate \eqref{20} with $f=(0,0)$. This completes the proof.\qed
\vspace{0,5cm}

Now let us return to Theorem \ref{t2}. In order to show \eqref{1} we need these two results.
\begin{prop}\label{p3} Let $n\geq3$  and let $1\leq \tilde{p},\tilde{q}\leq2$ and $0\leq \gamma\leq  1$ satisfy \eqref{1} with $1<\tilde{p}$ and $\tilde{q}'<\frac{2(n-1)}{n-3}$. Let $f\in\dot{\mathcal{H}}_{\gamma}({\R}^n)$, $F\in L^{\tilde{p}}_t(\R,L^{\tilde{q}}_x(\R^n))$ and let $u$ be the solution of \eqref{problem} with $\tau=0$ and $V=0$. Then, for all $r>0$, we have
\begin{equation}\label{29}\int_\R\norm{\left(e^{-r\left\langle x\right\rangle}u(t), e^{-r\left\langle x\right\rangle}u_t(t)\right)}^2_{\dot{\mathcal{H}}_\gamma({\R}^n)}\d t\lesssim\left(\norm{f}_{\dot{\mathcal{H}}_\gamma({\R}^n)}+\norm{F}_{L^{\tilde{p}}_t(\R,L^{\tilde{q}}_x(\R^n))}\right)^2.\end{equation}
\end{prop}
We prove Proposition \ref{p3} by combining Lemma \ref{l5} (for $\phi=e^{-r\left\langle x\right\rangle}$) with the arguments used in Proposition 2 of \cite{P2}.

\begin{prop}\label{p4} Assume $\rm(H1)$, $\rm(H2)$ and $\rm(H3)$ fulfilled. Let $n\geq3$  and let $1\leq \tilde{p},\tilde{q}\leq2$ and $0\leq \gamma\leq  1$ satisfy \eqref{1} with $1<\tilde{p}$ and $\tilde{q}'<\frac{2(n-1)}{n-3}$. Let $f\in\dot{\mathcal{H}}_{\gamma}({\R}^n)$, $F\in L^{\tilde{p}}_t(\R,L^{\tilde{q}}_x(\R^n))$ and let $u$ be the solution of \eqref{problem} with $\tau=0$. Then, we have
\begin{equation}\label{30}\int_0^{+\infty}\norm{\left(e^{-(D+2\epsilon)\left\langle x\right\rangle}u(t), e^{-(D+2\epsilon)\left\langle x\right\rangle}u_t(t)\right)}^2_{\dot{\mathcal{H}}_\gamma({\R}^n)}\d t\lesssim \left(\norm{f}_{\dot{\mathcal{H}}_\gamma({\R}^n)}+\norm{F}_{L^{\tilde{p}}_t(\R,L^{\tilde{q}}_x(\R^n))}\right)^2.\end{equation}
\end{prop}
Combining \eqref{20} with  Proposition \ref{p3}  (see also the proof of Corollary 2 in \cite{P2}), we prove easily \eqref{30}. \\
\ \\
\textit{Proof of Theorem \ref{t2}.}
We present the solution of \eqref{problem} as a sum $u=u_0+v$, where $u_0$ is the solution of
\[  \left\{\begin{aligned}
\partial_t^2u_0-\Delta_{x}u_0&=F(t,x),\ \ t>0,\ \ x\in{\R}^{n},\\
(u_0,\partial_tu_0)(0,x)&=(f_{1}(x),f_{2}(x))=f(x),\ \ x\in{\R}^n,\end{aligned}\right.\]
while $v$ is the solution of the problem
\[  \left\{\begin{aligned}
\partial_t^2v-\Delta_{x}v+V(t,x)v&=-Vu_0,\ \ t>0,\ \ x\in{\R}^{n},\\
(v,v_{t})(0,x)&=(0,0),\ \ x\in{\R}^n.\end{aligned}\right.\]
It  follows from \cite{KT}, that  for $u=u_0$ estimates \eqref{1} hold. It remains to estimate $v$. Next, we have
\[v(t)=-\int_0^t\frac{\sin((t-s)\Lambda)}{\Lambda}\left(V(s,\cdot)u_0(s)+V(s,\cdot)v(s)\right)\d s.\]
For $e^{\epsilon\left\langle x\right\rangle}Vu_0$ we have \eqref{29}. Moreover, (H1), \eqref{20} and \eqref{29} imply
\begin{equation}\label{31}\norm{e^{\epsilon\left\langle x\right\rangle}Vv}_{L_t^2\left(\R^+,\dot{{H}}^{\gamma}({\R}^n)\right)}\lesssim  \norm{e^{(\frac{3D}{2}+2\epsilon)\left\langle x\right\rangle}Vu_0}_{L_t^2\left(\R^+,\dot{{H}}^{\gamma}({\R}^n)\right)}
\lesssim\left(\norm{f}_{\dot{\mathcal{H}}_\gamma({\R}^n)}+\norm{F}_{L^{\tilde{p}}_t(\R,L^{\tilde{q}}_x(\R^n))}\right).\end{equation}
Then we have 
\begin{equation}\label{32}\norm{e^{\epsilon\left\langle x\right\rangle}\left(Vu_0+Vv\right)}_{L_t^2\left(\R^+,\dot{{H}}^{\gamma}({\R}^n)\right)}\lesssim\left(\norm{f}_{\dot{\mathcal{H}}_\gamma({\R}^n)}+\norm{F}_{L^{\tilde{p}}_t(\R,L^{\tilde{q}}_x(\R^n))}\right).\end{equation}
An application of Lemma \ref{l5} (with $\phi=e^{\epsilon\left\langle x\right\rangle}$) shows that the operator
\[S:\dot{H}^{-\gamma}(\R^n)\ni g\longmapsto e^{-\epsilon\left\langle x\right\rangle}e^{\pm it\Lambda}g\in L^2(\R^+,\dot{H}^{-\gamma}(\R^n))\]
is bounded. The adjoint operator
\[(S^*(G))(x)=\int_0^{+\infty} e^{\pm is\Lambda}e^{-\epsilon\left\langle x\right\rangle}G(s,x)\d s\]
is bounded as an operator from $L^2(\R^+,\dot{H}^{\gamma}(\R^n))$ to $\dot{H}^{\gamma}(\R^n)$ and this yields
\[\norm{\int_0^{+\infty} e^{\pm is\Lambda}e^{-\epsilon\left\langle x\right\rangle}h(s)}_{\dot{H}^{\gamma}(\R^n)}\lesssim\norm{h}_{L^2(\R^+,\dot{H}^{\gamma}(\R^n))}.\]
Combining this result with the arguments exploited in the last section of \cite{P2}, we obtain
\begin{equation}\label{33}\norm{\int_0^t\frac{\sin((t-s)\Lambda)}{\Lambda}\left(V(s,\cdot)u_0(s)+V(s,\cdot)v(s)\right)\d s}_{L^p_t(\R^+,L_x^q(\R^n))}\lesssim\norm{e^{\epsilon\left\langle x\right\rangle}\left(Vu_0+Vv\right)}_{L_t^2\left(\R^+,\dot{{H}}^{\gamma}({\R}^n)\right)}\end{equation}
and
\begin{equation}\label{34}\norm{\int_0^t\frac{\sin((t-s)\Lambda)}{\Lambda}\left(V(s,\cdot)u_0(s)+V(s,\cdot)v(s)\right)\d s}_{\dot{H}^{\gamma}(\R^n)}\lesssim\norm{e^{\epsilon\left\langle x\right\rangle}\left(Vu_0+Vv\right)}_{L_t^2\left(\R^+,\dot{{H}}^{\gamma}({\R}^n)\right)}.\end{equation}
Combining \eqref{32}, \eqref{33} and \eqref{34} we deduce that for $u=v$ estimates \eqref{1} hold. This completes the proof of Theorem \ref{t2}. \qed

\subsection{Examples of potential such that $\rm(H3)$ is fulfilled}
The purpose of this subsection is to give examples of potential $V(t,x)$ satisfying (H1) and (H2) such that (H3) is fulfilled. Let $T_0>0$ and let $(V_{T_1})_{T_1\geq T_0}$ be a set of functions lying in $\mathcal C^\infty(\R^{1+n})$ and satisfying (H1) such that $V_{T_1}(t,x)$ is $T_1$-periodic with respect to $t$ and
\[V(t,x)=V_0(x),\quad \textrm{for } T_0\leq t\leq T_1.\]
Now consider the equation 
\[  \left\{\begin{aligned}
\partial_t^2v-\Delta_{x}v+V_0(x)v&=0,\ \ t>0,\ \ x\in{\R}^{n},\\
(v,v_{t})(0,x)&=(f_1,f_2),\ \ x\in{\R}^n\end{aligned}\right.\]
and let $\mathcal V(t)$ be the associate propagator. Notice that for $V(t,x)=V_{T_1}(t,x)$ we have
\[U(t,s)=\mathcal V(t-s),\quad\textrm{for } T_0\leq t,s\leq T_1.\]
Assume that $V_0(x)$ is chosen such that we have the following estimate
\[\norm{e^{-D\left\langle x\right\rangle}\mathcal V(t)e^{-D\left\langle x\right\rangle}}_{\mathcal L(\B)}\lesssim p(t),\quad t>0,\]
where $p(t)$ satisfies \eqref{local1}.
We refer to  \cite{SZ} and \cite{Va} for sufficient conditions to obtain this estimate for stationary potentials (we can also choose $V_0(x)=0$). Now consider the following result.
\begin{lem}\label{taru} Assume $\rm(H1)$, $\rm(H2)$ fulfilled. Then, for all $0\leq t\leq T_0$,\[e^{D\left\langle x\right\rangle}\left(U(t,0)-\mathcal V(t)\right)e^{D\left\langle x\right\rangle}\in\mathcal L(\B)\] and  we have
\[\norm{e^{D\left\langle x\right\rangle}\left(U(t,0)-\mathcal V(t)\right)e^{D\left\langle x\right\rangle}}_{\mathcal L(\B)}\leq C(T_0),\]
where $C(T_0)$ depend only on $T_0$..
\end{lem}
\begin{proof}
An application of the Duhamel's principle yields
\[U(t,0)=\mathcal V(t)+\int_0^t\mathcal V(t-s)Q_1(s)U(s,0)\d s\]
with
\[Q_1(t)=\left(\begin{array}{cc}0&0\\ V_0(x)-V(t,x)&0\end{array}\right).\]
Thus, we conclude by applying Lemma \ref{l4} and (H1).\end{proof}

Following the arguments exploited in the last section of \cite{Ki3}, by applying Lemma \ref{taru}, one can show that for $T=T_1$ with $T_1$ sufficiently large  and for $V(t,x)=V_{T_1}(t,x)$ conditions (H1), (H2) and (H3) are fulfilled.

\end{document}